\documentclass[reqno,12pt,openbib]{amsart}

\usepackage{amsmath}
\usepackage{amssymb}
\usepackage{amsthm}
\usepackage{amscd}
\usepackage{mathrsfs}

\theoremstyle{plain}
\newtheorem{thm}{Theorem}[section]
\newtheorem{lem}[thm]{Lemma}
\newtheorem{prop}[thm]{Proposition}
\newtheorem{cor}[thm]{Corollary}
\theoremstyle{definition}
\newtheorem{defn}[thm]{Definition}

\newtheorem{rem}[thm]{Remark}

\newcommand{\ri}{\mathfrak{o}}
\newcommand{\mi}{\mathfrak{p}}

\newcommand{\rad}{\mathfrak{a}}
\newcommand{\seqrad}{\mathfrak{a}}

\newcommand{\Z}{\mathbf{Z}}
\newcommand{\C}{\mathbf{C}}
\newcommand{\He}{\mathcal{H}}

\newcommand{\Irr}{\mathrm{Irr}}
\newcommand{\rU}{\mathrm{U}}

\newcommand{\Lamdba}{\Lambda}

\newcommand{\ad}{\mathrm{ad}}
\newcommand{\Ad}{\mathrm{Ad}}

\newcommand{\e}{\sqrt{\varepsilon}}
\newcommand{\p}{\varpi}

\newcommand{\J}{\mathfrak{J}}

\title{Representations of non quasi-split unramified $\rU(4)$ over 
a $p$-adic field I:
Representations of non-integral level}
\author{Michitaka Miyauchi}
\address{
Department of Mathematics, Faculty of Science, Kyoto University, 
Oiwake Kita-Shirakawa Sakyo Kyoto 606-8502 JAPAN
}
\email{miyauchi@math.kyoto-u.ac.jp}
\begin{document}

\maketitle

\pagestyle{myheadings}
\markboth{MICHITAKA MIYAUCHI}{REPRESENTATIONS OF 
NON QUASI-SPLIT $\rU(4)$ OVER A $p$-ADIC FIELD I}

\begin{abstract}
Let $F_0$ be a non-archimedean local field of 
odd residual characteristic
and let $G$ be the non quasi-split unramified unitary group 
in four variables defined over $F_0$.
In this paper,
we give a classification
 of the irreducible smooth representations of $G$
of non-integral level
using the Hecke algebraic method developed by Howe and Moy.
\end{abstract}

\section*{Introduction}
Let $F_0$ be a non-archimedean local field of 
odd residual characteristic
and let $G$ be the non quasi-split unramified unitary group
in four variables defined over $F_0$.
Konno \cite{Konno}
classified the non supercuspidal representations 
of $G$ modulo supercuspidal representations of proper Levi
subgroups of $G$,
by the method which Sally and Tadi\'c
\cite{Sally_Tadic} used for $\mathrm{GSp}(4)$.
Note that her result does not depend on whether 
$G$ is unramified over $F_0$ or not.
In this paper,
we give a classification
 of the irreducible smooth representations of $G$
of non-integral level
using the Hecke algebraic approach 
developed by Moy \cite{GSp4} for $\mathrm{GSp}(4)$.

One of the purposes of this paper is 
a classification of the supercuspidal representations of $G$.
Stevens \cite{St5} proved that 
every irreducible supercuspidal representation of 
a $p$-adic classical group
is irreducibly induced from 
an open compact subgroup when $p$ is odd.
Although
the supercuspidal representations of $G$
are exhausted by his construction,
the classification has not been completed.

The other purpose of this research is 
Jacquet-Langlands correspondence
of $p$-adic unramified unitary groups in four variables.
By the philosophy of Langlands program,
there seems to be  a certain correspondence of 
discrete series representations 
between $G$ and unramified $\mathrm{U}(2,2)$ over $F_0$.
In this paper and \cite{U22},
we can find same phenomenons on supercuspidal representations
of those two groups.

In \cite{U21} and \cite{GSp4},
Moy gave a classification of 
the irreducible smooth representations 
of unramified $\rU(2,1)$ and $\mathrm{GSp}(4)$ over $F_0$,
based on the concepts of
nondegenerate representations
and Hecke algebra isomorphisms.
A nondegerate representation of $\mathrm{GSp}(4)$
is an irreducible representation $\sigma$ 
of an open compact subgroup 
which satisfies a certain cuspidality or semisimplicity condition.
An important property of nondegenerate representations 
of $\mathrm{GSp}(4)$ is that
every irreducible smooth representation of $\mathrm{GSp}(4)$
contains some nondegenerate representation.
For $\sigma$ a nondegenerate representation of
$\mathrm{GSp}(4)$,
the set of equivalence classes of irreducible representations
of $\mathrm{GSp}(4)$
which contain $\sigma$
can be identified with
the set of equivalence classes of irreducible representations of the Hecke algebra $\He$ associated
to $\sigma$.
Moy described $\He$ as a Hecke algebra of
some smaller group
and 
reduced the classification of 
the irreducible smooth representations of $\mathrm{GSp}(4)$ which contain
$\sigma$
to that of a smaller group.

In \cite{MP},
Moy and Prasad developed
the concept of nondegenate representations 
into that of unrefined minimal $K$-types for reductive 
$p$-adic groups.
For classical groups,
Stevens \cite{St3} gave an explicit construction of 
unrefined minimal $K$-types as fundamental skew strata
in terms of the lattice theory by Bushnell and Kutzko \cite{BK2}
and Morris \cite{Morris-1}.

Throughout this paper,
we use the notion of fundamental skew strata introduced by
\cite{St3} for our nondegenerate representations of $G$.
We give a brief summary of this
in Section \ref{sec:preliminaries}.
Let $F$ be the unramified quadratic extension over $F_0$.
Then $G$ is realized as the group of isometries
of an $F/F_0$-hermitian form on 4-dimensional 
$F$-vector space $V$.
According to \cite{St3},
a skew stratum 
is a 4-tuple $[\Lambda, n, r, \beta]$.
A periodic lattice function $\Lambda$ with a certain duality
induces
a filtration $\{P_{k}(\Lambda)\}_{k \geq 1}$ on 
a parahoric subgroup $P_{0}(\Lambda)$ of $G$.
Integers $n > r \geq 0$ 
and an element $\beta$ in the Lie algebra of $G$
determine a character $\psi$ of the group $P_{r+1}(\Lambda)$
which is trivial on $P_{n+1}(\Lambda)$.
Writing $e(\Lambda)$ for the period of $\Lambda$,
we refer to $n/e(\Lambda)$ as the level of the stratum.
A skew stratum $[\Lambda, n, r, \beta]$ is called fundamental
if $\beta$ is not nilpotent modulo $p$,
and semisimple if $\beta$ is semisimple and satisfy
a good property on the adjoint action.
Every irreducible smooth representation of $G$ of positive level
contains 
a character $\psi$ induced by a fundamental skew stratum,
and the level of the fundamental skew strata contained in 
an irreducible smooth representation $\pi$ of $G$ is an invariant of
$\pi$.
We refer to it as the level of $\pi$.

In section \ref{section:Moy},
we give a generalization of a result of Moy \cite{U21} and 
\cite{GSp4}.
For a skew semisimple stratum $[\Lambda, n, n-1, \beta]$
with tamely ramified algebra $F[\beta]$ over $F$,
we construct an irreducible representation $\rho$ of 
an open compact subgroup $J$ with the following two 
properties:

\begin{enumerate}
\item[(i)] An irreducible smooth representation of $G$
contains the skew stratum $[\Lambda, n, n-1, \beta]$
if and only if it contains $\rho$;

\item[(ii)] Writing $G_E$ for the $G$-centralizer of $\beta$,
the intertwining of $(J, \rho)$ equals to $J G_E J$.
For $g \in G_E$,
we have 
$J g J\cap G_E = (J\cap G_E)g (J\cap G_E)$.
\end{enumerate}

From the second property,
we guess the Hecke algebra associated to the pair $(J, \rho)$
is isomorphic to a certain Hecke algebra
of $G_E$.
It is true when $G_E$ is compact 
like in many cases in \cite{U21} and \cite{GSp4}.
But in general,
to construct Hecke algebra isomorphisms,
we need to know at least 
the structure of Hecke algebras of $p$-adic classical
groups associated to Iwahori higher congruence subgroups.

From Section~\ref{strata},
we start to classify the irreducible smooth representations of $G$.
First,
we prove 
a rigid result on the existence of fundamental skew strata,
which implies that this property holds for 
a set of skew strata with finite $G$-orbits modulo $p$.
Based on this,
we replace fundamental skew strata with semisimple ones,
case by case.

In Sections~\ref{sec:over4} and \ref{sec:r_3},
we classify representations of $G$ of level $n/4$
and $n/3$, respectively.
All of these representations are supercuspidal.
The algebra $F[\beta]$ of a semisimple stratum of level $n/3$ 
possesses a simple component whose ramification index is 3.
Although 
we can't apply the Moy's construction to this case when $p$ is 3,
the method by Stevens \cite{St2} and \cite{St1}
to construct supercuspidal representations from maximal compact tori does work well.

In Section~\ref{sec:h_int},
we give a classification of the irreducible smooth representations
of $G$ of half-integral level.
First,
we replace fundamental skew strata with semisimple ones.
According to the form of the centralizer $G_E$,
there are three kinds of such strata.
In two cases, $G_E$ is compact
and in the other case $G_E$ is isomorphic to 
a product of unramified
$\rU(1,1)$ over $F_0$ and unramified $\rU(1)$ over a quadratic ramified 
extension over $F_0$.
In the non compact case,
we establish a Hecke algebra isomorphism 
by checking the relations of elements of the corresponding Hecke algebra.

In Section~\ref{comparison},
we compare semisimple skew strata for $G$
of non-integral level 
with those for unramified $\rU(2,2)$ over $F_0$
in \cite{U22}.
In the point of view of Hecke algebraic method,
behavior of 
supercuspidal representations of those two groups of non-integral 
level are quite similar.
In fact,
the value of level and characteristic polynomials are same
and the difference of the form of $G_E$ relates to 
inner forms.

\section{Preliminaries}\setcounter{equation}{0}\label{sec:preliminaries}
In this section, we recall the notion of fundamental skew strata
for $p$-adic classical groups
from \cite{BK2} and \cite{St2}.
We refer the reader to those papers 
for more details.
\subsection{Filtrations}
Let $F$ be a non-archimedean local field
of odd residual characteristic equipped with a 
(possibly trivial) Galois involution ${}^-$,
and let $\ri_F$ denote the ring of integers in $F$,
$\mi_F$ the maximal ideal in $\ri_F$,
$k_F = \ri_F/\mi_F$ the residue field.

Let $F_0$ denote the ${}^-$-fixed subfield of $F$.
We denote by
$\ri_0$, $\mi_0$, $k_0$ the objects for $F_0$ 
analogous to those above for $F$,
and by $q$ the number of elements in $k_0$.

We select a uniformizer $\p_F$ in $F$ so that
$\overline{\p_F} = -\p_F$ if $F/F_0$ is ramified.
Otherwise we take $\p_F$ in $F_0$.

Let $V$ be an $N$-dimensional $F$-vector space 
equipped with a nondegenerate hermitian form $h$
with respect to $F/F_0$.
We put
$A = \mathrm{End}_F(V)$,
$\widetilde{G} = A^\times$
and denote by
$\sigma$ the involution on $A$ induced by $h$.
We also put
$G= \{ g \in \widetilde{G}\ |\ g\sigma(g) = 1 \}$,
the corresponding classical group over $F_0$,
$A_- = \{ X \in A\ |\ X+\sigma(X)=0 \}
\simeq \mathrm{Lie}(G)$.

Recall from \cite{BK2} (2.1) that
an $\ri_F$-lattice sequence in $V$ is a function $\Lambda$ from $\mathbf{Z}$
to the set of $\ri_F$-lattices in $V$ such that
\begin{enumerate}
\item[(i)] $\Lambda(i) \supset \Lambda(i+1)$ for all $i \in \mathbf{Z}$;

\item[(ii)]
there exists an integer $e(\Lambda)$ 
such that
$\p_F \Lambda(i) = \Lambda(i+e(\Lambda))$ for all $i \in \mathbf{Z}$.
\end{enumerate}
\noindent
The integer $e(\Lambda)$
is called the $\ri_F$-period of $\Lambda$.
We say that
an $\ri_F$-lattice sequence $\Lambda$ is strict 
if 
$\Lambda(i) \supsetneq \Lambda(i+1)$ for all $i \in \Z$.

For $L$ an $\ri_F$-lattice in $V$,
we define its dual lattice $L^\#$ by
$L^\# = \{ v \in V\ |\ h(v, L) \subset \ri_F\}$.
Recall from \cite{St2} Section 3 that
an $\ri_F$-lattice sequence $\Lambda$ in $V$ is called
self-dual if there exists an integer $d(\Lambda)$
such that
$\Lambda(i)^\# = \Lambda(d(\Lambda)-i)$
for all $i \in \Z$. 

Recall from \cite{Kariyama}  \S 1.4 that
a $C$-sequence in $V$ is 
a self-dual $\ri_F$-lattice sequence $\Lambda$ in $V$
which satisfies
\begin{enumerate}
\item[$C$(i)] 
$\Lambda(2i+1) \supsetneq \Lambda(2i+2)$  for all $i \in \Z$,

\item[$C$(ii)] $e(\Lambda)$ is even and $d(\Lambda)$ is odd.
\end{enumerate}
\noindent
We remark that 
this is a realization of a $C$-chain in \cite{Morris-1} \S 4.3
as an $\ri_F$-lattice sequence.

An $\ri_F$-lattice sequence $\Lambda$ in $V$
induces
a filtration $\{\rad_n(\Lambda)\}_{n \in \Z}$ on $A$ by
\begin{eqnarray*}
\mathfrak{a}_n(\Lambda)
= \{ X \in A\ |\ X\Lambda(i) \subset \Lambda(i+n)\ \mathrm{for\ all}\ i \in 
\mathbf{Z}\},\ n \in \Z.
\end{eqnarray*}
Note that an $\ri_F$-lattice sequence 
$\Lambda$ in $V$ is self-dual if and only if
$\sigma(\seqrad_n(\Lambda)) = \seqrad_n(\Lambda)$,
$n \in \Z$. 
This filtration 
determines a kind of \lq\lq valuation"
$\nu_\Lambda$ on $A$ by
\[
\nu_\Lambda(x) = \sup\{ n \in \Z\ |\ x \in \seqrad_n(\Lambda)\},\ 
x \in A\backslash \{0\},
\]
with the usual understanding that
$\nu_\Lambda(0) = \infty$.

Let $\Lambda$ be an $\ri_F$-lattice sequence in $V$.
For $k \in \Z$,
we define an $\ri_F$-lattice sequence $\Lambda +k$ by
$(\Lambda +k)(i) = \Lambda(i+k)$, $i \in \Z$.
Then we have $\seqrad_n(\Lambda) = \seqrad_{n}(\Lambda +k)$, $n \in \Z$.
We refer to $\Lambda + k$ as a translate of $\Lambda$.
For $g \in G$,
we define an $\ri_F$-lattice sequence $g\Lambda$ by
$(g\Lambda)(i) = g\Lambda(i)$, $i \in \Z$.
Note that if $\Lambda$ is self-dual,
then $\Lambda + k$ and $g \Lambda$ are also self-dual.

For $\Gamma$ an $\ri_F$-lattice in $A$,
we define its dual $\Gamma^*$ by
$\Gamma^*  =  \{ X \in A\ |\ \mathrm{tr}_{A/F_0}(X\Gamma) \subset \mi_0\}$,
where $\mathrm{tr}_{A/F_0}$ denotes the composition of traces
$\mathrm{tr}_{F/F_0} \circ \mathrm{tr}_{A/F}$.
Recall from \cite{BK2} (2.10) that,
if $\Lambda$ is an $\ri_F$-lattice sequence in $V$,
then we have 
$\seqrad_n(\Lambda)^* = \seqrad_{1-n}(\Lambda)$,
$n \in \Z$.

For $S$ a subset of $G$, 
we write $S_- = S \cap A_-$.
Let $\Lambda$ be a self-dual $\ri_F$-lattice sequence
in $V$.
We  set
$P_{0}(\Lambda) = G\cap \seqrad_0(\Lambda)$ and
$P_{n}(\Lambda) = G \cap (1 + \seqrad_n(\Lambda))$,
for $n \in \Z$, $n \geq 1$.

We fix an additive character
$\psi_0$ of $F_0$ with conductor $\mi_0$.
Let 
${}^\wedge$ denote the Pontrjagin dual.
For $x$ a real number, 
we write $[x]$ for the greatest integer less than or equal to $x$.
\begin{prop}\label{thm:Morris}
\label{prop:p}
Let $\Lambda$ be a self-dual $\ri_F$-lattice sequence in  
$V$ and 
let 
$n, r \in \Z$ satisfy
$n > r \geq [n/2] \geq 0$.
Then
the map $x \mapsto x-1$ induces an
isomorphism
$P_{r+1}(\Lambda)/P_{n+1}(\Lambda) \simeq 
\rad_{r+1}(\Lambda)_-/\rad_{n+1}(\Lambda)_-$ and
there exists an isomorphism of finite abelian groups
\[
\rad_{-n}(\Lambda)_-/\rad_{-r}(\Lambda)_- \simeq
(P_{r+1}(\Lambda)/P_{n+1}(\Lambda))^\wedge;
b+ \rad_{-r}(\Lambda)_- \mapsto \psi_b,
\]
where
$\psi_b(x) =
\psi_0( \mathrm{tr}_{A/F_0}(b(x-1)))$, 
$x \in P_{r+1}(\Lambda)$.
\end{prop}
\subsection{Skew strata}
\begin{defn}[\cite{BK2} (3.1), \cite{St2} Definition 4.5]
(i) A stratum in $A$ is a 4-tuple $[\Lambda, n, r, \beta]$
consisting of an $\ri_F$-lattice sequence $\Lambda$ in $V$,
integers $n, r $ such that $n > r \geq 0$, 
and an element $\beta$ in $\seqrad_{-n}(\Lambda)$.
We say that two strata $[\Lambda, n, r, \beta_i]$, $i = 1,2$,
are equivalent if $\beta_1 \equiv \beta_2 
\pmod{\seqrad_{-r}(\Lambda)}$.

(ii) A stratum $[\Lambda, n, r, \beta]$ in $A$ is called skew
if $\Lambda$ is self-dual and $\beta \in \rad_{-n}(\Lambda)_-$.
\end{defn}

The fraction 
$n/e(\Lambda)$ is called the level of the stratum.
If $n > r \geq [n/2]$,
then by
Proposition~\ref{thm:Morris},
an equivalence class of skew strata $[\Lambda, n, r, \beta]$ corresponds to a character $\psi_\beta$ of
$P_{r+1}(\Lambda)/P_{n+1}(\Lambda)$.

For $g \in \widetilde{G}$ and $x \in A$,
we write $\Ad(g)(x) = gxg^{-1}$.
We define the formal intertwining of a skew stratum
$[\Lambda, n, r, \beta]$ to be
\[
I_{G}[\Lambda, n, r, \beta]
= \{ g \in G\ |\ (\beta + 
\rad_{-r}(\Lambda)_-)
\cap \Ad(g)(\beta + \rad_{-r}(\Lambda)_-) \neq \emptyset \}.
\]
If $n > r \geq [n/2]$, then it is just
the intertwining of the character $\psi_\beta$ of $P_{r+1}(\Lambda)$
in $G$.

For
$[\Lambda, n, n-1, \beta]$ a stratum in $A$,
we set
$y_\beta = \p_F^{n/k} \beta^{e(\Lambda)/k} \in \seqrad_0(\Lambda)$,
where
$k = (e(\Lambda), n)$.
Then the characteristic polynomial 
$\Phi_\beta(X)$ of $y_\beta$
lies in $\ri_F[X]$.
We define the characteristic polynomial 
$\phi_\beta(X) \in k_F[X]$ of the stratum 
to be the reduction  modulo $\mi_F$ of $\Phi_\beta(X)$.
\begin{defn}[\cite{BK1} (2.3)]
A
stratum $[\Lambda, n, r, \beta]$ in $A$ is called fundamental 
if
$\phi_\beta(X) \neq X^N$.
\end{defn}

Let 
$\pi$ be a smooth representation of $G$
and $[\Lambda, n, r, \beta]$
a skew stratum with $n > r \geq [n/2]$.
We say that $\pi$ contains 
$[\Lambda, n, r, \beta]$ if 
the restriction of $\pi$ to 
$P_{r+1}(\Lambda)$ contains the corresponding character
$\psi_\beta$.
A smooth representation $\pi$ of $G$
is called of positive level
if $\pi$ has no non-zero $P_{1}(\Lambda)-$fixed vector,
for $\Lambda$ any self-dual $\ri_F$-lattice sequence in $V$.

Let $\pi$ be an irreducible smooth representation of $G$
of positive level.
Due to \cite{St3} Theorem 2.11,
$\pi$
contains a fundamental skew stratum $[\Lambda, n, n-1, \beta]$.
Let $[\Lambda, n, n-1, \beta]$ be any skew stratum 
contained in $\pi$.
Then by the philosophy of minimal $K$-types,
$[\Lambda, n, n-1, \beta]$ is fundamental
if and only if the level $n/e(\Lambda)$ of the stratum 
is smallest among in those of skew strata occurring in $\pi$.
Thus we can define the level of $\pi$ by 
the level of the fundamental skew strata contained in $\pi$.

Recall that when skew strata $[\Lambda, n, n-1, \beta]$
and $[\Lambda', n', n'-1, \beta']$
are contained in an irreducible smooth representation of $G$,
there exists $g \in G$ such that
\begin{eqnarray}
(\beta + \rad_{1-n}(\Lambda)_-)\cap 
\Ad(g)(\beta' + \rad_{1-n'}(\Lambda')_-) \neq \emptyset.
\end{eqnarray}
We therefore see that
for an irreducible smooth representation $\pi$ of $G$ of positive level,
the characteristic polynomial of a fundamental skew stratum 
contained in $\pi$ depends only on $\pi$.
We refer to it as the characteristic polynomial of $\pi$.

\subsection{Semisimple strata}
\begin{defn}[\cite{BK1} (1.5.5), \cite{BK2} (5.1)]
A stratum 
$[\Lambda, n, r, \beta]$ in $A$ is called simple if
\begin{enumerate}
\item[(i)] the algebra $E = F[\beta]$ is a field,
and $\Lambda$ is an $\ri_E$-lattice sequence;

\item[(ii)] $\nu_\Lambda(\beta) = -n$;

\item[(iii)] $\beta$ is minimal over $F$ in the sense of \cite{BK1} p. 41.
\end{enumerate}
\end{defn}
\begin{rem}
The definition above is a special case of 
that in \cite{BK1}.
With the notion of \cite{BK1},
our simple stratum 
is a simple stratum $[\Lambda, n, r, \beta]$
with $k_0(\beta, \Lambda) = -n$.
\end{rem}

The following criterion of the simplicity of strata
is well known.
\begin{prop}[\cite{char_l} Proposition 1.5]\label{prop:pure}
Let
$\Lambda$ be a strict $\ri_F$-lattice sequence in $V$ 
with $e(\Lambda) = N$
and let $n$ be an integer such that $(n, N) = 1$.
Suppose that
$[\Lambda, n, r, \beta]$ is a fundamental stratum in $A$.
Then $F[\beta]$ is a totally ramified extension of degree $N$
over $F$ and
$[\Lambda, n, r, \beta]$ is simple.
\end{prop}

Let $[\Lambda, n, r, \beta]$ be a stratum in $A$.
We assume that there is a non-trivial $F$-splitting $V = V^1 \oplus V^2$ such that
\begin{enumerate}
\item[(i)]
$\Lambda(i) = \Lambda^1(i) \oplus \Lambda^2(i)$ for all
$i \in \Z$,
where $\Lambda^j(i) = \Lambda(i) \cap V^j$, for $j = 1, 2$;

\item[(ii)] $\beta V^j \subset V^j$ for $j =1,2$.
\end{enumerate}
\noindent
For $j = 1,2$, we write $\beta_j = \beta|_{V^j}$.
By \cite{BK2} (2.9),  we get
a stratum $[\Lambda^j, n, r, \beta_j]$
 in $\mathrm{End}_F(V^j)$.
Recall from \cite{BK2} (3.6) that
a stratum $[\Lambda, n, r, \beta]$ in $A$ is called split if
\begin{enumerate}
\item[(iii)] $\nu_{\Lambda^1}(\beta_1) = -n$ and
$X$ does not divide $\phi_{\beta_1}(X)$;

\item[(iv)] either $\nu_{\Lambda^2}(\beta_2) > -n$, or else all the following 
conditions hold:
\begin{enumerate}
\item[(a)] $\nu_{\Lambda^2}(\beta_2) = -n$ and
$X$ does not divide $\phi_{\beta_2}(X)$,

\item[(b)] $(\phi_{\beta_1}(X), \phi_{\beta_2}(X)) = 1$.
\end{enumerate}
\end{enumerate}
\begin{defn}[\cite{St2} Definition 4.8, \cite{St3} Definition 2.10]
(i) (Inductive definition on the dimension of $V$)
 A stratum $[\Lambda, n, r, \beta]$ is called semisimple
if it is simple, or else
it is split as above
and satisfies the following conditions:
\begin{enumerate}
\item[(a)] $[\Lambda^1, n, r, \beta_1]$ is simple;

\item[(b)]
$[\Lambda^2, n_2, r, \beta_2]$ is semisimple 
or $\beta_2 = 0$, where
$n_2 = \mathrm{max}\{-\nu_{\Lambda^2}(\beta_2),
r+1\}$.
\end{enumerate}

(ii)
A skew stratum in $A$ is called split
if it is split with respect to an orthogonal $F$-splitting
$V = V^1 \bot V^2$.
\end{defn}

If a skew stratum $[\Lambda, n, r, \beta]$ is split
with respect to $V = V^1 \bot V^2$,
then 
$\Lambda^j$ is a self-dual $\ri_F$-lattice sequence
in $(V^j, h|_{V^j})$ with $d(\Lambda^j) = d(\Lambda)$,
for $j =1, 2$.
\begin{rem}
Let $[\Lambda, n, r, \beta]$ be a skew semisimple stratum in $A$.
Then the $G$-centralizer $G_E$ of $\beta$ is a 
product of classical groups over extensions of $F_0$.

We claim that if $F/F_0$ is quadratic unramified,
then $G_E$ is a product of unramified unitary groups.
It suffices to prove this in the simple case.

Suppose that $[\Lambda, n, n-1, \beta]$ is a skew simple stratum.
Since $F/F_0$ is unramified, there is $\varepsilon \in \ri_0^\times$
such that $F = F_0[\e]$.
Let $E_0$ denote the $\sigma$-fixed subfield of $E$.
Then $E/E_0$ is a quadratic extension
and $G_E$ is the unitary group over $E_0$,
corresponding to the involutive algebra
$(\mathrm{End}_E(V), \sigma)$.
Suppose that $E/E_0$ is ramified.
We can take a uniformizer $\p_E$ of $E$ so that $\sigma(\p_E)
= -\p_E$.
Then the element $\p_E \e \in E_0$ is a uniformizer of $E$.
This contradicts the assumption.
\end{rem}

\subsection{Representations associated maximal tori}
\label{sec:max}
We recall from \cite{St2} and \cite{St1}
the construction of supercuspidal representations 
associated to maximal compact tori.

Let $[\Lambda, n, n-1, \beta]$ be a skew semisimple
stratum
such that the $G$-centralizer $G_E$ of $\beta$ is 
a maximal torus.
Then $G_E$ lies in $P_{0}(\Lambda)$ and 
by \cite{St2} Theorem 4.6,
we get 
\begin{eqnarray}\label{eq:intertwining}
I_G[\Lambda, n, [n/2], \beta] = G_E P_{[(n+1)/2]}(\Lambda).
\end{eqnarray}

Put $J = G_E P_{[(n+1)/2]}(\Lambda)$, 
$J^1 = P_{1}(\Lambda_{\ri_E}) P_{[(n+1)/2]}(\Lambda)$,
and
$H^1 = P_{1}(\Lambda_{\ri_E}) P_{[n/2]+1}(\Lambda)$,
where $P_1(\Lambda_{\ri_E}) = G_E \cap P_1(\Lambda)$.
The character $\psi_\beta$ of $P_{[n/2]+1}(\Lamdba)$ 
can extend to 
a character $\theta$ of $H^1$
since $H^1/P_{[n/2]+1}(\Lambda)$ is abelian.
Thank to \cite{St1} Proposition 4.1,
there exists a unique irreducible representation $\eta_\theta$
of $J^1$
which contains $\theta$.
Moreover the restriction of $\eta_\theta$ to $H^1$
is a multiple of $\theta$
and $\dim \eta_\theta = [J^1: H^1]^{1/2}$,
which is a power of $q$.

The order of the finite abelian group
$J/J^1 \simeq G_E/P_1(\Lambda_{\ri_E})$
is coprime to $q$,
and hence
$\eta_\theta$ can extend an irreducible representation $\kappa_\theta$ of $J$.
For any $\chi \in 
(G_E /P_1(\Lambda_{\ri_E}))^\wedge$,
we have 
$I_G(\chi \otimes \kappa_\theta) = I_G(\psi_\beta) = J$,
so that the compactly induced representation
$\mathrm{Ind}_J^G (\chi \otimes \kappa_\theta)$
is irreducible and supercuspidal.

It is easy to check that 
every irreducible smooth representation of $G$
which contains $[\Lambda, n, [n/2], \beta]$
can be constructed in this way.

We classify these representations.
Let $\theta$ and $\theta'$ be extensions of $\psi_\beta$
to $H^1$.
Let $\kappa_\theta$ and $\kappa'_{\theta'}$ be as above.
Let $\chi$  and $\chi'$ be characters of $G_E /P_1(\Lambda_{\ri_E})$.
Suppose that 
$\mathrm{Ind}_J^G (\chi \otimes \kappa_\theta)
\simeq 
\mathrm{Ind}_J^G (\chi' \otimes \kappa'_{\theta'})$.
Then there is $g \in G$ which intertwines 
$\chi \otimes \kappa_\theta$
and $\chi' \otimes \kappa'_{\theta'}$.
Thinking of  the restriction to $P_{[n/2]+1}(\Lambda)$,
we obtain $g \in I_G(\psi_\beta) = J$,
and hence
$\chi \otimes \kappa_\theta \simeq 
\chi' \otimes \kappa'_{\theta'}$.
Restricting it to $H^1$,
we get $\theta = \theta'$.
The representation $\kappa'_{\theta'}$
is isomorphic to $\chi'' \otimes \kappa_\theta$,
for some $\chi'' \in 
(G_E /P_1(\Lambda_{\ri_E}))^\wedge$,
and we get $\chi = \chi' \chi''$.

\begin{rem}
The number of the irreducible smooth representations of 
$G$ containing $[\Lambda, n, [n/2], \beta]$
equals to
$[G_E: G_E \cap P_{[n/2]+1}(\Lambda)]$.
\end{rem}

\subsection{Uniqueness of lattice sequences}\label{sec:recover}
Let $[\Lambda, n, n-1, \beta]$ be a skew semisimple stratum
associated to an orthogonal $F$-splitting
$V = V^1 \bot \ldots \bot V^k$.
Then $\Lambda = \Lambda^1 \bot \ldots \bot \Lambda^k$
and $\beta = \beta_1 + \cdots + \beta_k$,
where $\Lambda^i(j) = \Lambda(j)\cap V^i$, $j \in \Z$ and 
$\beta_i = \beta|_{V^i}$, for $1 \leq i \leq k$.
If we write $e = e(\Lambda) = e(\Lambda^i)$ and
$k = (n, e)$,
then we have $\Phi_\beta(X) = \prod_i \Phi_{i}(X)$
where $\Phi_{i}(X)$ is the characteristic polynomial 
of $y_i = \p_F^{n/k}\beta_i^{e/k} \in \mathrm{End}_F(V^i)$.

By the definition of split strata,
$\Phi_i(X)$ and $\Phi_j(X)$ are coprime modulo $\mi_F$,
for $i \neq j$.
Note that if $[\Lambda^i, n, n-1, \beta_i]$ is simple,
then $\Phi_i(X) \bmod{\mi_F}$ is a power of an irreducible polynomial
in $k_F$.
If $\beta_i = 0$, then we have $\Phi_i(X) \bmod{\mi_F}= X^{N_i}$,
where $N_i = \dim_F V^i$.
So we see that $V^i$ is just 
the kernel of $\Phi_i(y_\beta)$, for $1 \leq i \leq k$.
In particular, $y_\beta$ determines the $F$-splitting $V = V_1 \bot \ldots \bot V_k$ uniquely.

The algebra
$E_i = F[\beta_i]$ is an extension over $F$ and 
$\Lambda^i$ is an $\ri_{E_i}$-lattice sequence in $V^i$.
We write $B_i$ for the $\mathrm{End}_F(V^i)$-centralizer of $\beta_i$.
Then the involutive algebra $(B_i, \sigma)$ defines
a nondegenerate hermitian form $h_{E_i}$ on the $E_i$-space $V^i$
up to scalar in $E_i^\times$.
For $1 \leq i \leq k$,
since $\rad_j(\Lambda) \cap B_i$ is $\sigma$ stable for all $j \in \Z$,
the sequence
$\Lambda^i$ is a self-dual $\ri_{E_i}$-lattice sequence
in $(V^i, h_{E_i})$.
\begin{lem}\label{lem:recover}
Suppose that the space $(V^i, h_{E_i})$ is anisotropic for all $1 \leq i \leq k$.
Then we can recover $\Lambda$
from $\beta$, $n$, $e(\Lambda)$ and $d(\Lambda)$.
\end{lem}
\begin{proof}
The assumption implies that 
$\Lambda^i$ is a unique self-dual $\ri_{E_i}$-lattice 
sequence in $(V_i, h_{E_i})$ of $\ri_F$-period $e(\Lambda)$,
up to translation.
Recall that $\Lambda^i$ is a self-dual $\ri_F$-lattice sequence
in $(V^i, h|_{V_i})$ with $d(\Lambda^i) = d(\Lambda)$.
So
$d(\Lambda)$ determines $\Lambda^i$, for all $1 \leq i \leq k$.
Therefore we can recover $\Lambda$ by the equation
$\Lambda = \Lambda^1 \bot \ldots 
\bot \Lambda^k$.
\end{proof}
\subsection{Hecke algebras}
Let $G$ be a unimodular, 
locally compact, totally disconnected topological group, 
$J$ an open compact subgroup of $G$,
and $(\sigma, W)$ an irreducible smooth representation of $J$.
For $g \in G$, we write
$\sigma^g$ for the representation of
$J^g = g^{-1}Jg$ defined by
$\sigma^g (x) = \sigma(g x g^{-1})$, $x \in J^g$.
We define the intertwining of $\sigma$ in $G$ by
\[
I_G(\sigma) = 
\{ g \in {G}\ 
|\ \mathrm{Hom}_{J\cap J^g}(\sigma, \sigma^g) \neq 0 \}.
\]

Let $(\widetilde{\sigma}, \widetilde{W})$ denote 
the contragradient representation of $(\sigma, W)$.
The Hecke algebra $\He(G//J, \sigma)$ 
is the set of 
compactly supported functions $f : G  \rightarrow \mathrm{End}_{\C}(\widetilde{W})$
such that
\[
f(k g k') = \widetilde{\sigma}(k)f(g) \widetilde{\sigma}(k'),\
k, k' \in J,\ g \in G.
\]
Let $dg$ denote the Haar measure on $G$ normalized
so that the volume $\mathrm{vol}(J)$ of $J$ is $1$.
Then $\He(G//J, \sigma)$ becomes an algebra
under convolution relative to $dg$.
Recall from \cite{BK1} (4.1.1) that the support 
of $\He(G//J, \sigma)$
is the intertwining of  $\widetilde{\sigma}$ in $G$,
that is, 
\[
I_G(\widetilde{\sigma}) = 
\bigcup_{f\in \He(G//J, \sigma)}\mathrm{supp}(f).
\]

Since $J$ is compact,
there exists a $J$-invariant, positive definite hermitian form 
on $\widetilde{W}$.
This form induces an involution $X \mapsto \overline{X}$
on $\mathrm{End}_{\C}(\widetilde{W})$.
For $f \in \He(G//J, \sigma)$,
we define $f^* \in \He(G//J, \sigma)$ by 
$f^*(g) = \overline{f(g^{-1})},\ g \in G$.
Then the map $*: \He(G//J, \sigma) \rightarrow 
\He(G//J, \sigma)$ is an involution on $\He(G//J, \sigma)$.

Let 
$\Irr(G)$ denote the set of equivalence classes of 
irreducible smooth representations of $G$
and $\Irr(G)^{(J, \sigma)}$
the subset of $\Irr(G)$ consisting of
elements
whose $\sigma$-isotypic components are not zero.
Let $\Irr \He(G//J, \sigma)$ denote the 
set of equivalence classes of irreducible representations
of $\He(G//J, \sigma)$.
Then, by \cite{BK1} (4.2.5), there is a bijection 
$\Irr(G)^{(J, \sigma)} \simeq \Irr \He(G//J, \sigma)$.

\section{A generalization of a result of Moy}\label{section:Moy}
Let $[\Lambda, n, n-1, \beta]$ be a  skew semisimple stratum 
in $A$ associated to an orthogonal $F$-splitting 
$V = V^1 \bot \ldots \bot V^k$.
As usual, we write $\beta = \beta_1 + \cdots + \beta_k$,
where $\beta_i = \beta|_{V^i}$.
We have $E = F[\beta] = \bigoplus_{1 \leq i \leq k} E_i$,
where $E_i = F[\beta_i]$.

Throughout this section,
we 
assume that $E_i$ is tamely ramified over $F$ for all $1 \leq i \leq k$.

We put $A^{ij} = \mathrm{Hom}_F(V^j, V^i)$, for $1 \leq i, j \leq k$.
When we write $B_i$ for the $A^{ii}$-centralizer of $\beta_i$,
the $A$-centralizer $B$ of $\beta$ equals to 
$\bigoplus_{1 \leq i \leq k}B_i$.
Let $B^\bot$ denote the orthogonal complement of $B$
in $A$ with respect to the pairing induced by
$\mathrm{tr}_{A/F}$,
and let $B_i^\bot$ denote that of $B_i$
with respect to $\mathrm{tr}_{A^{ii}/F}$.
Then 
we have 
\begin{eqnarray}
B^\bot = \bigoplus_{i \neq j} A^{ij} \oplus \bigoplus_i B_i^\bot.
\end{eqnarray}
and $A = B \oplus B^\bot$.
Note that
the set $B$ and $B^\bot$
are $\sigma$-stable since $\beta \in A_-$.
\begin{prop}\label{prop:decomp_E}
For $k \in \Z$,
we have
$\rad_k(\Lambda) = \rad_k(\Lambda)\cap B \oplus 
\rad_k(\Lambda)\cap B^\bot$.
\end{prop}
\begin{proof}
By \cite{BK2} Proposition 2.9, we have
$\rad_k(\Lambda) 
= \bigoplus_{i, j} \rad_k(\Lambda)\cap A^{ij}$
and $\rad_k(\Lambda)\cap A^{ii} = \rad_k(\Lambda^i)$,
for $1 \leq i \leq k$.

Suppose that $[\Lambda^i, n, n-1, \beta_i]$ is simple.
It follows from \cite{BK1} Remark (1.3.8) (ii) that
the $(B_i, B_i)$-bimodule projection $s_i: A^{ii} \rightarrow B_i$
with kernel $B_i^\bot$
satisfies
$s_i(\rad_k(\Lambda^i)) = \rad_k(\Lambda)\cap B$
because we are assuming that $E_i$ is tamely ramified over $F$.
Hence
we get $\rad_k(\Lambda)\cap A^{ii}
= \rad_k(\Lambda)\cap B_i \oplus
\rad_k(\Lambda)\cap B_i^\bot$.
If $\beta_i = 0$,
then $B_i = A^{ii}$.
This completes the proof.
\end{proof}

For 
$k \in \Z$,
we abbreviate
$\rad_k = \rad_k(\Lambda)$,
$\rad'_k =  \rad_k(\Lambda) \cap B$ and
$\rad_k^\bot = \rad_k(\Lambda) \cap B^\bot$.
Define $\sigma$-stable $\ri_F$-lattices 
$\J$ and $\J_+$ in $A$ by
\begin{eqnarray}\label{eq:def_J}
\J = \rad'_n \oplus \rad_{[(n+1)/2]}^\bot,\
\J_+ = \rad'_n \oplus \rad_{[n/2]+1}^\bot,
\end{eqnarray}
and open compact subgroups $J$ and $J_+$ of  $G$ by
\begin{eqnarray}
J = (1+\J)\cap G,\
J_+ = (1 + \J_+)\cap G,
\end{eqnarray}
as in \cite{GSp4} (4.16).

Since $J_+ \subset P_{[n/2]+1}(\Lambda)$,
the quotient $J_+/P_{n+1}(\Lambda)$ is abelian.
As usual,
we get an isomorphism of finite abelian groups
\begin{eqnarray*}
(\rad_{-n})_-/ (\J_+^*)_- \simeq (J_+/P_{n+1}(\Lambda))^\wedge;\
b + (\J_+^*)_- \mapsto \Psi_b,
\end{eqnarray*}
where
\begin{eqnarray}\label{eq:Psi}
\Psi_b(p) = \psi_0(\mathrm{tr}_{A/F_0}(b(p-1))),\ p \in J_+.
\end{eqnarray}
Due to Proposition~\ref{prop:decomp_E} and \cite{BK2} (2.10),
we get
\begin{eqnarray}
\J_+^* = \rad_{1-n}' \oplus 
\rad_{-[n/2]}^\bot.
\end{eqnarray}

For $X \in A$,
we write $\mathrm{ad}(\beta)(X) = \beta X -X\beta$. 
Since $\beta \in (\rad_{-n})_-$,
the map $\ad(\beta)$ induces a quotient map
$\ad(\beta):
\rad_{k}^\bot/\rad_{k+1}^\bot \rightarrow 
\rad_{k-n}^\bot /\rad_{k-n+1}^\bot$, 
for $k \in \Z$.

\begin{lem}\label{lem:adjoint}
For $k \in \Z$,
the map 
$\ad(\beta):
\rad_{k}^\bot/\rad_{k+1}^\bot \rightarrow 
\rad_{k-n}^\bot /\rad_{k-n+1}^\bot$ 
is an isomorphism.
\end{lem}
\begin{proof}
By the periodicity of the filtration $\{\rad_k(\Lambda)\}$,
it suffices to prove that the induced map is injective for all $k \in \Z$.

Recall that 
$\rad_l^\bot = \bigoplus_{i \neq j}\rad_l \cap A^{ij}
\oplus \bigoplus_{i} \rad_l \cap B_i^\bot$, for $l \in \Z$.
By \cite{BK2} \S 3.7 Lemma 2,
$\ad(\beta)$ maps $\rad_k \cap A^{ij}$
onto $\rad_{k-n} \cap A^{ij}$, for $i \neq j$.
Thus the assertion is reduced to the simple case.

Let $[\Lambda, n, n-1, \beta]$ be a simple stratum in $A$.
Let $e$ denote the $\ri_E$-period of $\Lambda$
and let $V'$ be an $e$-dimensional $E$-vector space.
Then there is a strict $\ri_E$-lattice sequence $\Lambda'$ 
in $V'$ of $\ri_E$-period $e$.
Define a strict $\ri_E$-lattice sequence $\Lambda''$
in $V'' = V \oplus V'$ by 
$\Lamdba''(i) = \Lambda(i) \oplus \Lambda'(i)$,  $i \in \Z$.
We confuse $\beta$ with $\beta \cdot 1_{V''}$.
Then we get a simple stratum $[\Lambda'', n, n-1, \beta]$
in $A'' = \mathrm{End}_F(V'')$.

Let $B''$ denote the $A''$-centralizer of $\beta$.
It follows from \cite{BK1} (1.4.9) that 
if $x \in \rad_0(\Lambda'')$ satisfies
$\ad(\beta)(x) \in \rad_{r-n}(\Lambda'')$, for $r \geq 1$,
then we have $x \in B'' + \rad_{r}(\Lambda'')$.
For $k \geq 0$,
we see that if $x \in \rad_{k}(\Lambda)\cap B^\bot$
satisfies $\ad(\beta)(x) \in \rad_{k-n+1}(\Lambda)$, 
then 
$x$ lies in $(B'' + \rad_{k+1}(\Lambda''))\cap B^\bot
= \rad_{k+1}(\Lambda)\cap B^\bot$,
by \cite{BK2} Proposition 2.9.
By the periodicity of $\{\rad_i(\Lambda)\}_{i \in \Z}$,
this holds for all $k \in \Z$.
This completes the proof.
\end{proof}

Since $\beta$ is skew,
we obtain the following corollary:
\begin{cor}\label{cor:adjoint}
For $k \in \Z$,
the map 
$\ad(\beta):
(\rad_{k}^\bot)_-/(\rad_{k+1}^\bot)_- \rightarrow 
(\rad_{k-n}^\bot)_- /(\rad_{k-n+1}^\bot)_-$ 
is an isomorphism.
\end{cor}

\begin{prop}[\cite{GSp4} Lemma 4.4]\label{prop:5c1_surj}
Suppose that an element 
$\gamma \in \beta + (\rad_{1-n})_-$ 
lies in $B_-$ 
modulo $(\rad_{k-n})_-$
for some integer $k \geq 1$.
Then, 
there exists 
$p \in P_{k}(\Lambda)$ such that
$\mathrm{Ad}(p)(\gamma)  \in  \beta + (\rad'_{1-n})_-$.
\end{prop}
\begin{proof}
Exactly the same as the proof of \cite{GSp4} Lemma 4.4.
\end{proof}

As an immediate corollary of the proof we have
\begin{cor}[\cite{GSp4} Corollary 4.5]\label{cor:5c1_J}
$\Ad(J)(\beta + \rad'_{1-n}) = \beta + (\J_+^*)_-$.
\end{cor}

\begin{prop}[\cite{GSp4} Theorem 4.1]\label{prop:5_2}
Let $\pi$ be an irreducible smooth representation of $G$.
Then $\pi$ contains $[\Lambda, n, n-1, \beta]$ if and only if 
$\pi$ contains $(J_+, \Psi_\beta)$.
\end{prop}
\begin{proof}
Since $\Psi_\beta$ is an extension of $\psi_\beta$ to $J_+$,
it is obvious that
if $\pi$ contains $\Psi_\beta$,
then $\pi$ contains $\psi_\beta$.

Suppose that $\pi$ contains $[\Lambda, n, n-1, \beta]$.
Then $\pi$ contains an extension of $\psi_\beta$ to 
$P_{[n/2]+1}(\Lambda)$.
This extension has the form $\psi_\gamma$, for some
$\gamma \in \beta + (\rad_{1-n})_-$.
By Proposition~\ref{prop:5c1_surj},
replacing $\psi_\gamma$ with a $P_1(\Lambda)$-conjugate,
we may assume 
$\gamma \in \beta + (\rad_{1-n}')_-$.
Then the restriction of $\psi_\gamma$ to $J_+$ is equal to 
$\Psi_\beta$.
This completes the proof.
\end{proof}

We put $G_E = G \cap B$.
\begin{prop}\label{prop:E_surj}
$I_G(\Psi_\beta) = J G_E J$.
\end{prop}
\begin{proof}
We see that an element $g \in G$ lies in 
$I_G(\Psi)$ if and only if
$\mathrm{Ad}(g)(\beta + (\J_+^*)_-)\cap 
(\beta + (\J_+^*)_-) \neq \emptyset$ .
We have
$J I_G(\Psi) J = I_G(\Psi)$
and $G_E \subset I_G(\Psi)$,
and hence 
$J G_E J \subset I_G(\Psi)$.

Let $g \in I_G(\Psi)$.
Due to Corollary~\ref{cor:5c1_J},
there exists an element $k$ in $JgJ$
such that
$\Ad(k)(\beta+(\rad_{1-n}')_-)\cap (\beta +
(\rad_{1-n}')_-) \neq \emptyset$.
Take $x, y \in (\rad_{1-n}')_-$
so that
$\ad(\beta)(k) = kx-yk$.
If we write $k^\bot$ for the $B^\bot$-component
of $k$,
then we get $\ad(\beta)(k^\bot) = k^\bot x-yk^\bot$.
Suppose $k^\bot \in \rad_{l}$, for some $l \in \Z$.
Then we have
$\ad(\beta)(k^\bot) \in \rad_{l-n+1}$,
and hence by  Lemma~\ref{lem:adjoint},
$k^\bot \in \rad_{l+1}$.
This implies $k^\bot = 0$
and hence $k \in G_E$.
This completes the proof.
\end{proof}

Since $[J, J] \subset  P_n(\Lambda) \subset J_+$,
we can define an alternating form $\theta$ on $J/J_+$ by
\begin{eqnarray}
\theta(x, y) = \Psi_\beta([x, y]),\ x, y \in J.
\end{eqnarray}
\begin{lem}\label{lem:nondeg}
The form $\theta$ is nondegenerate.
\end{lem}
\begin{proof}
Let $1 + x$ be an element in $J$ such  that
$\Psi_\beta([1 + x, 1 + y]) = 1$, for all $1 + y \in J$.
Since 
\begin{eqnarray*}
\Psi_\beta([1+x, 1+y]) & = & 
\psi_0(\mathrm{tr}_{A/F_0}(\beta(xy-yx)))\\
& = & 
\psi_0(\mathrm{tr}_{A/F_0}(\ad(\beta)(x)y)),
\end{eqnarray*}
we obtain 
$\ad(\beta)(x) \in \J^*_- = (\rad_{1-n}' \oplus \rad_{1-[(n+1)/2]}^\bot)\cap A_-$.
If we write $x^\bot$ for the $B^\bot$-part of $x$,
then we have
$x^\bot \in \rad_{[(n+1)/2]}^\bot$ and 
$\ad(\beta)(x) \in \rad_{1-[(n+1)/2]}^\bot$.
Lemma~\ref{lem:adjoint} implies that 
$x^\bot$ lies in $\rad_{[n/2]+1}^\bot$.
So we get $1 +x \in J_+$.
This completes the proof.
\end{proof}

It follows from Lemma~\ref{lem:nondeg} that
there exists a unique irreducible representation $\rho$ of $J$
which contains $\Psi_\beta$.
Moreover the restriction of $\rho$ to $J_+$ is a multiple
of $\Psi_\beta$
and $\dim \rho = [J:J_+]^{1/2}$,
which is a power of $q$.

As direct consequences of Propositions~\ref{prop:5_2} and
\ref{prop:E_surj},
we get the following two propositions.
\begin{prop}\label{prop:rho}
An irreducible smooth representation $\pi$ of $G$
contains $[\Lambda, n, n-1, \beta]$ if and only if
$\pi$ contains $\rho$.
\end{prop}
\begin{prop}\label{prop:H_surj2}
$I_G(\rho) = J G_E J$.
\end{prop}

We put $J' = G_E \cap J$.
Then we have $J' = G_E \cap J_+ = P_n(\Lambda)\cap B$.
\begin{prop}\label{prop:H_inj}
For $g \in G_E$, we have
$JgJ \cap G_E = J'gJ'$.
\end{prop}
\begin{proof}
Suppose that $[\Lambda, n, n-1, \beta]$ is simple.
Let $\Lambda'$ be a strict $\ri_F$-lattice sequence 
such that $\Lambda'(\Z) = \Lambda(\Z)$.
Then $\Lambda'$ is also an $\ri_E$-lattice sequence.
We write $e'$ for the $\ri_E$-period of $\Lambda'$ and 
$\nu_E$ for the normalized valuation on $E$.
Thus we have $\rad_n(\Lambda) = \rad_{n'}(\Lambda')$,
where $n' = -e' \nu_E(\beta)$.
Due to \cite{BK1} Theorem (1.6.1),
we obtain
$(1 + \rad_n(\Lambda))x (1+\rad_n(\Lambda)) \cap B
= (1 + \rad_n(\Lambda)\cap B)x (1+\rad_n(\Lambda)\cap B)$,
for $x \in B^\times$.

Now the proof is exactly same as that of \cite{St5} Lemma 2.6.
\end{proof}

For $x$ in $A$,
we denote by $x'$ its $B$-component
and by $x^\bot$ its $B^\bot$-component.
The next lemma is useful to check some relations on 
$\He(G//J, \rho)$.
\begin{lem}\label{lem:intertwine_J}
For $g \in JG_EJ$,
we have
$\nu_\Lambda(g^\bot) \geq \nu_\Lambda(g') +[(n+1)/2]$.
\end{lem}
\begin{proof}
Put $k = \nu_\Lambda(g)$.
Then, for any element $y$ in $JgJ$,
we have
$y \equiv g \pmod{\rad_{k+[(n+1)/2]}}$,
so that
$y \equiv g^\bot \pmod{B+ \rad_{k +[(n+1)/2]}}$.
Therefore if $g \in JG_EJ$,
then $g^\bot \in \rad_{k +[(n+1)/2]}^\bot$.
In particular, we have
$\nu_\Lambda(g^\bot) \geq \nu_\Lambda(g) +[(n+1)/2] 
> \nu_{\Lambda}(g)$, and hence
$\nu_\Lambda(g') = \nu_\Lambda(g)$.
This completes the proof.
\end{proof}

The contragradient representation $\widetilde{\rho}$ of $\rho$
is the unique irreducible representation of $J$ which contains
$\widetilde{\Psi_\beta} = \Psi_{-\beta}$.
Suppose that $G_E$ is compact.
Then we have $G_E \subset P_0(\Lambda)$
and $G_E$ is the Iwahori subgroup of itself.
The oscillator representation yields a representation $\omega$
of $G_E/
(G_E \cap P_1(\Lambda))$ on the space of $\widetilde{\rho}$
with the property
\begin{eqnarray}
\omega(g) \widetilde{\rho}(p) \omega(g^{-1})
= \widetilde{\rho}(Ad(g)(p)),\ p \in J,\ g \in G_E.
\end{eqnarray}

For $g \in G_E$,
let $f_g$ denote the element in $\He(G//J, \rho)$
such that $f_g(g) = \omega(g)$ and $\mathrm{supp}(f_g) = Jg J$,
and 
let $e_g$ denote the element in $\He(G_E//J', \Psi_\beta)$
such that $e_g(g) = 1$ and $\mathrm{supp}(e_g) = J'g J'$.
Then we obtain the following
\begin{thm}\label{thm:Hecke}
With the notations as above,
suppose that $G_E$ is compact.
Then the map
$\eta:
\He(G_E//J', \Psi_\beta) \rightarrow \He(G//J, \rho)$
defined by
$\eta(e_g) = f_g$
is a support preserving, $*$-isomorphism.
\end{thm}
\begin{proof}
It is obvious that $\eta$ is an algebra homomorphism
which preserves supports.
Since $g \in G_E$ normalizes $(J_+, \Psi_\beta)$,
we have $\rho \simeq \rho^g$.
Hence the algebra $\He(G//J, \rho)$ is spanned by 
$f_g$, $g \in G_E$.
Then Propositions~\ref{prop:H_surj2} and \ref{prop:H_inj}
implies that 
$\eta$ is an isomorphism.
The $*$-preservation follows from 
$e_g^* = e_{g^{-1}}$ and $f_g^* = f_{g^{-1}}$,
for $g \in G_E$.
\end{proof}

\section{Fundamental strata for non 
quasi-split $\rU(4)$}
\setcounter{equation}{0}\label{strata}
\subsection{Non quasi-split $\rU(4)$}
From now on,
we assume that $F/F_0$ is quadratic unramified.
Let $\varepsilon$ denote a non-square unit in $\ri_0$.
Then we have $F = F_0[\e]$.
We can (and do) take a common uniformizer $\p$ of 
$F$ and $F_0$.

Let $V = F^4$ denote
the four dimensional $F$-space of column vectors.
We put $A = M_4(F)$ and $\widetilde{G} = A^\times$.
Let $\{e_i\}_{1 \leq i \leq 4}$ denote the standard $F$-basis 
of $V$,
and $E_{i j}$ the element in $A$ whose $(k, l)$
entry is $\delta_{ik}\delta_{jl}$.
We put
\begin{eqnarray}
H = E_{14} + \p E_{22} + E_{33} + E_{41} \in A
\end{eqnarray}
and define a nondegenerate $F/F_0$-hermitian form 
$h$ on $V$ by
$h(v, w) = {}^t \overline{v} H w$, for  $v, w \in V$.
Then $h$ induces 
an involution $\sigma$ on $A$ by the formula
$\sigma(X) = H^{-1} {}^t\overline{X}H$, for $X \in A$.
For $X = (X_{ij})\in A$,
we have
\begin{eqnarray}
\sigma(X) = 
\left(
\begin{array}{cccc}
\overline{X_{44}} & \p \overline{X_{24}} & 
\overline{X_{34}} & \overline{X_{14}}\\
\p^{-1}\overline{X_{42}} & \overline{X_{22}} & 
\p^{-1}\overline{X_{32}} & \p^{-1} \overline{X_{12}}\\
\overline{X_{43}} & \p \overline{X_{23}} &
\overline{X_{33}} & \overline{X_{13}} \\
\overline{X_{41}} & \p \overline{X_{21}} &
\overline{X_{31}} & \overline{X_{11}}
\end{array}
\right).
\end{eqnarray}

We put
$G = \{ g \in \widetilde{G}\ |\ \sigma(g) = g^{-1} \}$
and
$A_- = \{ X \in A\ |\ \sigma(X) = -X\} \simeq 
\mathrm{Lie}(G)$.
Then $G$ is the non quasi-split unramified 
unitary group in four variables
defined over $F_0$,
and $A_-$ consists of matrices of the form
\begin{eqnarray}
 \left(
\begin{array}{cccc}
Z & \p C & D & a\sqrt{\varepsilon}\\
M & b\e & Y & -\overline{C} \\ 
N & -\p \overline{Y} & c\e & -\overline{D} \\
d\e & -\p\overline{M} & -\overline{N} & -\overline{Z}
\end{array}
\right),\ 
C, D, M, N, Y, Z \in F,\ a,b,c,d \in F_0.
\end{eqnarray}
\subsection{A version of the existence of fundamental strata}
Define $\ri_F$-lattices $N_0$ and $N_1$ in $V$ by
\begin{eqnarray*}
& N_0 = \ri_F e_{1} \oplus \ri_F e_{2} \oplus \ri_F e_3 \oplus \ri_F e_4,\
N_1 = \ri_F e_{1} \oplus \ri_F e_{2} \oplus \ri_F e_3 \oplus \mi_F e_4.
\end{eqnarray*}
Then we have
\begin{eqnarray*}
& N_0^\# = \ri_F e_{1} \oplus \mi_F^{-1} e_{2} \oplus \ri_F e_3 \oplus \ri_F e_4,\
N_1^\# = \mi_F^{-1} e_{1} \oplus \mi_F^{-1} e_{2} \oplus \ri_F e_3 \oplus \ri_F e_4
\end{eqnarray*}
and obtain the following sequence of $\ri_F$-lattices in $V$:
\begin{eqnarray*}
\ldots \supsetneq N_0 \supsetneq N_1 \supsetneq \p N_1^\# \supsetneq
\p N_0^\# \supsetneq \p N_0 \supsetneq \ldots.
\end{eqnarray*}
Recall from
\cite{Morris-2} that
a self-dual $\ri_F$-lattice sequence $\Lambda$ in $V$
is called standard if 
$\Lambda(\Z) = \{\Lambda(i)\ |\ i \in \Z\}$ 
is 
contained in the set
$\{ \p^m N_0,\ \p^m N_1,\ \p^m N_0^\#,\ \p^m N_1^\#\ |\ m \in \Z\}$.
By {\it op. cit.} Proposition 1.10,
every self-dual $\ri_F$-lattice sequence 
is a $G$-conjugate of a standard one.

Let $\Lambda$ be a $C$-sequence in $V$
and $L$ an $\ri_F$-lattice in $V$.
Since $d(\Lambda)$ is odd,
we see that
$L \in \Lambda(2\Z)$
if and only if $L^\# \in \Lambda(2\Z+1)$.
So it is easy to observe that
there are just the following 8 standard  $C$-sequences 
$\Lambda_i$, $1 \leq i \leq 8$ in $V$,
up to translation:

(I) $C$-sequences with $\Lambda(2\Z) \cap \Lambda(2\Z+1)
= \emptyset$:
\begin{eqnarray}\label{eq:st_20}
\Lambda_1(2i) = \p^i N_0,\
\Lambda_1(2i+1) = \p^{i+1} N_0^\#,\ i \in \Z;
\end{eqnarray}
\begin{eqnarray}\label{eq:st_21}
\Lambda_2(2i) = \p^i N_1,\ 
\Lambda_2(2i+1) = \p^{i+1} N_1^\#,\ i \in \Z;
\end{eqnarray}
\begin{eqnarray*}
\Lambda_3(4i) = \p^i N_0,\
\Lambda_3(4i+1) = \p^i N_1,
\end{eqnarray*}
\begin{eqnarray}\label{eq:st_4}
\Lambda_3(4i+2) = \p^{i+1} N_1^\#,\ 
\Lambda_3(4i+3) = \p^{i+1} N_0^\#,\ i \in \Z.
\end{eqnarray}
(II) $C$-sequences with 
$\Lambda(2\Z)\cap \Lamdba(2\Z+1) \neq \emptyset$
and $\Lambda(2\Z) \neq \Lambda(2\Z+1)$:
\begin{eqnarray*}
\Lambda_4(6i) = \Lambda_4(6i+1)= \p^i N_0,\
\Lambda_4(6i+2) = \p^i N_1,
\end{eqnarray*}
\begin{eqnarray}\label{eq:c0}
\Lambda_4(6i+3) = \p^{i+1} N_1^\#,\ 
\Lambda_4(6i+4) = \Lambda_4(6i+5)= \p^{i+1} N_0^\#,\ i \in \Z;
\end{eqnarray}
\begin{eqnarray*}
\Lambda_5(6i-1)  =\p^i N_0,\
\Lambda_5(6i) = \Lambda_5(6i+1)= \p^i N_1,
\end{eqnarray*}
\begin{eqnarray}\label{eq:c1}
\Lambda_5(6i+2) = \Lambda_5(6i+3) = \p^{i+1} N_1^\#,\ 
\Lambda_5(6i+4) = \p^{i+1} N_0^\#,\ i \in \Z.
\end{eqnarray}
(III) $C$-sequences with $\Lambda(2\Z) = \Lambda(2\Z+1)$:
\begin{eqnarray}\label{eq:ns_20}
\Lambda_6(4i)= \Lambda_6(4i+1) = \p^i N_0,\
\Lambda_6(4i+2) = \Lambda_6(4i+3) = \p^{i+1} N_0^\#,\ i \in \Z;
\end{eqnarray}
\begin{eqnarray}\label{eq:ns_21}
\Lambda_7(4i)= \Lambda_7(4i+1) = \p^i N_1,\
\Lambda_7(4i+2) = \Lambda_7(4i+3) = \p^{i+1} N_1^\#,\ i \in \Z;
\end{eqnarray}
\begin{eqnarray*}
\Lambda_8(8i) = \Lambda_8(8i+1) = \p^i N_0,\
\Lambda_8(8i+2) = \Lambda_8(8i+3) = \p^i N_1,
\end{eqnarray*}
\begin{eqnarray}\label{eq:nt_4}
\Lambda_8(8i+4) = \Lambda_8(8i+5) = \p^{i+1} N_1^\#,\ 
\Lambda_8(8i+6) = \Lambda_8(8i+7) = \p^{i+1} N_0^\#,\ i \in \Z.
\end{eqnarray}
\begin{thm}\label{thm:strict}
Let $\pi$ be an irreducible smooth 
representation of $G$ of positive 
level.
Then $\pi$ contains a fundamental skew stratum
$[\Lambda, n, n-1, \beta]$ which satisfies one of the 
following conditions:
\begin{enumerate}
\item[$(i)$] $\Lambda = \Lambda_i$, for some $1 \leq i \leq 5$ 
and
$(e(\Lambda), n) =2$;

\item[$(ii)$] $\Lambda = \Lambda_i$, for some $1 \leq i \leq 3$
and $(e(\Lambda), n) = 1$.
\end{enumerate}
\end{thm}
\begin{proof}
Let $\pi$ be an irreducible smooth representation
of $G$ of positive level.
Thanks to \cite{Kariyama} Proposition 3.1.1,
$\pi$ 
contains a fundamental skew stratum $[\Lambda, n, n-1, \beta]$ such that
$\Lambda$ is a $C$-sequence and 
$(e(\Lambda), n) = 2$.
After $G$-conjugation,
we may assume $\Lambda$ is 
one of $\Lambda_i$, $1 \leq i \leq 8$.
If $\Lambda = \Lambda_i$, for $1 \leq i \leq 5$,
then there is nothing left to prove.

Suppose $\Lambda = \Lambda_i$, for $6 \leq i \leq 8$.
Put $\Lambda' = \Lambda_{i-5}$.
Then $\Lambda$ is the double of $\Lambda'$,
whence
$\rad_k(\Lambda') = \rad_{2k-1}(\Lambda) = \rad_{2k}(\Lambda)$, $k \in \Z$.
Since 
skew strata $[\Lambda, n, n-1, \beta]$ and
$[\Lambda', n/2, n/2-1, \beta]$ correspond to 
the same character $\psi_\beta$ of the group
$P_{n}(\Lambda) = P_{n/2}(\Lambda')$,
$\pi$ contains a fundamental skew stratum $[\Lambda', n/2, n/2-1, \beta]$,
which satisfies the condition (ii).
\end{proof}

We list up $\Lambda$ and $n$ of the
fundamental strata $[\Lambda, n, n-1, \beta]$
satisfying one 
of the conditions in Theorem~\ref{thm:strict}.
\begin{eqnarray}
\begin{array}{|c|c|c|c|c|}\hline
\Lambda & e(\Lambda) & d(\Lambda) & n & 
n/e(\Lambda)   \\
\hline
\Lambda_1 & 2 & -1 & 2m & m  \\
\hline
\Lambda_2 & 2 & -1 & 2m & m   \\
\hline
\Lambda_3 & 4 & -1 & 4m-2 & m-1/2  \\
\hline
\Lambda_4 & 6 &  -1 &
\begin{array}{c}6m-2\\ 6m-4\end{array} & 
\begin{array}{c}m-1/3\\ m-2/3\end{array}  \\
\hline
\Lambda_5 & 6 & -3 & 
\begin{array}{c}6m-2\\ 6m-4\end{array} & 
\begin{array}{c}m-1/3\\ m-2/3\end{array}  \\
\hline
\Lambda_1 & 2 & -1 & 2m-1 & m-1/2  \\
\hline
\Lambda_2 & 2 & -1 &  2m-1 & m-1/2  \\
\hline
\Lambda_3 & 4 & -1 & 
\begin{array}{c}4m-1\\ 4m-3\end{array} & 
\begin{array}{c}m-1/4\\ m-3/4\end{array} \\
\hline
\end{array}
\end{eqnarray}

\subsection{Filtrations}
We give an explicit description of
the filtrations on $A$ induces by 
standard $C$-sequences $\Lambda_i$, for $1 \leq i \leq 5$.
Since $\{\rad_k(\Lambda_i)\}_{k \in \Z}$ is periodic,
it suffices to describe
$\rad_k(\Lambda_i)$, $0 \leq k \leq e(\Lambda_i)-1$.

The sequences $\Lambda_1$ and $\Lambda_2$
correspond to the standard filtrations of 
maximal compact subgroups
of $G$:
\begin{eqnarray}
\rad_0(\Lambda_1) = 
\left(
\begin{array}{c|c|cc}
\ri_F & \mi_F & \ri_F & \ri_F\\ \hline
\ri_F & \ri_F & \ri_F & \ri_F\\ \hline
\ri_F & \mi_F & \ri_F & \ri_F\\ 
\ri_F & \mi_F & \ri_F & \ri_F\\ 
\end{array}
\right),\
\rad_1(\Lambda_1) = 
\left(
\begin{array}{c|c|cc}
\mi_F & \mi_F & \mi_F & \mi_F\\ \hline
\ri_F & \mi_F & \ri_F & \ri_F\\ \hline
\mi_F & \mi_F & \mi_F & \mi_F\\ 
\mi_F & \mi_F & \mi_F & \mi_F\\ 
\end{array}
\right);
\end{eqnarray}
\begin{eqnarray}
\rad_0(\Lambda_2) = 
\left(
\begin{array}{cc|c|c}
\ri_F & \ri_F & \ri_F & \mi_F^{-1}\\ 
\ri_F & \ri_F & \ri_F & \mi_F^{-1}\\ \hline
\mi_F & \mi_F & \ri_F & \ri_F\\  \hline
\mi_F & \mi_F & \mi_F & \ri_F\\ 
\end{array}
\right),\
\rad_1(\Lambda_2) = 
\left(
\begin{array}{cc|c|c}
\mi_F & \mi_F & \ri_F & \ri_F\\ 
\mi_F & \mi_F & \ri_F & \ri_F\\ \hline
\mi_F & \mi_F & \mi_F & \ri_F\\  \hline
\mi_F^2 & \mi_F^2 & \mi_F & \mi_F\\ 
\end{array}
\right).
\end{eqnarray}

The sequence $\Lambda_3$
corresponds to the standard filtration of 
the Iwahori subgroup of $G$:
\begin{eqnarray*}
\rad_0(\Lambda_3) = 
\left(
\begin{array}{cc|cc}
\ri_F & \mi_F & \ri_F & \ri_F\\ 
\ri_F & \ri_F & \ri_F & \ri_F\\ \hline
\mi_F & \mi_F & \ri_F & \ri_F\\ 
\mi_F & \mi_F & \mi_F & \ri_F\\ 
\end{array}
\right),\
\rad_1(\Lambda_3) = 
\left(
\begin{array}{cc|cc}
\mi_F & \mi_F & \ri_F & \ri_F\\ 
\ri_F & \mi_F & \ri_F & \ri_F\\ \hline
\mi_F & \mi_F & \mi_F & \ri_F\\ 
\mi_F & \mi_F & \mi_F & \mi_F\\ 
\end{array}
\right),\
\end{eqnarray*}
\begin{eqnarray}
\rad_2(\Lambda_3) = 
\left(
\begin{array}{cc|cc}
\mi_F & \mi_F & \mi_F & \ri_F\\ 
\mi_F & \mi_F & \ri_F & \ri_F\\ \hline
\mi_F & \mi_F & \mi_F & \mi_F\\ 
\mi_F & \mi_F^2 & \mi_F & \mi_F\\ 
\end{array}
\right),\
\rad_3(\Lambda_3) = 
\left(
\begin{array}{cc|cc}
\mi_F & \mi_F & \mi_F & \mi_F\\ 
\mi_F & \mi_F & \mi_F & \ri_F\\ \hline
\mi_F & \mi_F^2 & \mi_F & \mi_F\\ 
\mi_F^2 & \mi_F^2 & \mi_F & \mi_F\\ 
\end{array}
\right).
\end{eqnarray}

Sequences $\Lambda_4$ and $\Lambda_5$ give
non-standard filtrations of the Iwahori subgroup
of $G$:
\begin{eqnarray*}
\rad_0(\Lambda_4) = \rad_0(\Lambda_5) = 
\rad_0(\Lambda_3),\ 
\rad_1(\Lambda_4) = \rad_1(\Lambda_5) = 
\rad_1(\Lambda_3),
\end{eqnarray*}
\begin{eqnarray*}
\rad_2(\Lambda_4) = 
\left(
\begin{array}{cc|cc}
\mi_F & \mi_F & \mi_F & \ri_F\\ 
\ri_F & \mi_F & \ri_F & \ri_F\\ \hline
\mi_F & \mi_F & \mi_F & \mi_F\\ 
\mi_F & \mi_F & \mi_F & \mi_F\\ 
\end{array}
\right),\
\rad_3(\Lambda_4) = 
\left(
\begin{array}{cc|cc}
\mi_F & \mi_F & \mi_F & \mi_F\\ 
\mi_F & \mi_F & \ri_F & \ri_F\\ \hline
\mi_F & \mi_F & \mi_F & \mi_F\\ 
\mi_F & \mi_F^2 & \mi_F & \mi_F\\ 
\end{array}
\right),\
\end{eqnarray*}
\begin{eqnarray}\label{eq:filtration_4}
\rad_4(\Lambda_4) = 
\left(
\begin{array}{cc|cc}
\mi_F & \mi_F & \mi_F & \mi_F\\ 
\mi_F & \mi_F & \mi_F & \ri_F\\ \hline
\mi_F & \mi_F^2 & \mi_F & \mi_F\\ 
\mi_F & \mi_F^2 & \mi_F & \mi_F\\ 
\end{array}
\right),\
\rad_5(\Lambda_4) = 
\left(
\begin{array}{cc|cc}
\mi_F & \mi_F^2 & \mi_F & \mi_F\\ 
\mi_F & \mi_F & \mi_F & \mi_F\\ \hline
\mi_F & \mi_F^2 & \mi_F & \mi_F\\ 
\mi_F^2 & \mi_F^2 & \mi_F & \mi_F\\ 
\end{array}
\right),\
\end{eqnarray}
\begin{eqnarray*}
\rad_2(\Lambda_5) = 
\left(
\begin{array}{cc|cc}
\mi_F & \mi_F & \ri_F & \ri_F\\ 
\mi_F & \mi_F & \ri_F & \ri_F\\ \hline
\mi_F & \mi_F & \mi_F & \ri_F\\ 
\mi_F & \mi_F^2 & \mi_F & \mi_F\\ 
\end{array}
\right),\
\rad_3(\Lambda_5) = 
\left(
\begin{array}{cc|cc}
\mi_F & \mi_F & \mi_F & \ri_F\\ 
\mi_F & \mi_F & \ri_F & \ri_F\\ \hline
\mi_F & \mi_F & \mi_F & \mi_F\\ 
\mi_F^2 & \mi_F^2 & \mi_F & \mi_F\\ 
\end{array}
\right),\
\end{eqnarray*}
\begin{eqnarray}
\rad_4(\Lambda_5) = 
\left(
\begin{array}{cc|cc}
\mi_F & \mi_F & \mi_F & \ri_F\\ 
\mi_F & \mi_F & \mi_F & \ri_F\\ \hline
\mi_F & \mi_F^2 & \mi_F & \mi_F\\ 
\mi_F^2 & \mi_F^2 & \mi_F & \mi_F\\ 
\end{array}
\right),\
\rad_5(\Lambda_5) = 
\left(
\begin{array}{cc|cc}
\mi_F & \mi_F & \mi_F & \mi_F\\ 
\mi_F & \mi_F & \mi_F & \ri_F\\ \hline
\mi_F^2 & \mi_F^2 & \mi_F & \mi_F\\ 
\mi_F^2 & \mi_F^2 & \mi_F^2 & \mi_F\\ 
\end{array}
\right).
\end{eqnarray}

\section{Representations of level $n/4$}\label{sec:over4}
\setcounter{equation}{0}
Let $n$ be a positive integer such that $(n, 4) = 1$.
In this section,
we give a classification of  
the irreducible smooth representations of $G$
of level $n/4$.
Theorem~\ref{thm:strict} says that
such a representation contains 
a fundamental skew stratum $[\Lambda_3, n, [n/2], \beta]$.
By Proposition~\ref{prop:pure},
$[\Lambda_3, n, n-1, \beta]$ is simple
and $E = F[\beta]$ is a totally ramified extension 
of degree 4 over $F$.

\begin{prop}\label{prop:conj_4}
Let 
$[\Lambda_3, n, [n/2], \beta]$ and
$[\Lambda_3, n, [n/2], \gamma]$ be fundamental skew strata
contained in some irreducible smooth representation of $G$.
Then 
$(P_{[n/2]+1}(\Lambda_3), \psi_\gamma)$ is 
a $P_{0}(\Lambda_3)$-conjugate of 
$(P_{[n/2]+1}(\Lambda_3), \psi_\beta)$.
\end{prop}
\begin{proof}
By assumption, there is $g \in G$ such that
\begin{eqnarray*}
(\beta + \rad_{-[n/2]}(\Lambda_3)_-)\cap 
\Ad(g)(\gamma + \rad_{-[n/2]}(\Lambda_3)_-)\neq \emptyset.
\end{eqnarray*}
Take $\delta
\in (\beta + \rad_{-[n/2]}(\Lambda_3)_-)\cap 
\Ad(g)(\gamma + \rad_{-[n/2]}(\Lambda_3)_-)$.
Then we get simple strata
$[\Lambda_3, n, [n/2], \delta]$
and $[g\Lambda_3, n, [n/2], \delta]$.
Recall that $d(\Lambda_3) = d(g\Lambda_3)$.
By Lemma~\ref{lem:recover},
we have $\Lambda_3 = g\Lambda_3$ and 
hence $g \in P_{0}(\Lambda_3)$.
This completes the proof.
\end{proof}

By Proposition~\ref{prop:conj_4},
the set of equivalence classes of level $n/4$ representations
of $G$ is the disjoint union
$\bigcup_\beta \Irr(G)^{(P_{[n/2]+1}(\Lambda_3), \psi_\beta)}$,
where $\beta$ runs over the $P_{0}(\Lambda_3)$-conjugacy 
classes of elements in $\rad_{-n}(\Lambda_3)_-
/\rad_{-[n/2]}(\Lambda_3)_-$
such that $[\Lambda_3, n, n-1, \beta]$ is fundamental.
For each $\beta$,
the set $\Irr(G)^{(P_{[n/2]+1}(\Lambda_3), \psi_\beta)}$
is classified via the result in \S \ref{sec:max}.

\section{Representations of level $n/3$}\label{sec:r_3}
Let $n$ be a positive integer such that $(n, 3) = 1$.
In this section,
we classify the irreducible smooth representations of $G$
of level $n/3$.
It follows from Theorem~\ref{thm:strict}
that every irreducible smooth representation of $G$
of level $n/3$ contains 
a fundamental skew stratum $[\Lambda, 2n, n, \beta]$
such that $\Lambda = \Lambda_4$ or $\Lambda_5$.

Let $\Lambda = \Lambda_4$
and let $[\Lambda, 2n, n, \beta]$ be a fundamental 
skew stratum.
The characteristic polynomial $\phi_\beta(X)$ depends only on 
the coset $\beta + \rad_{1-2n}(\Lambda)_-$.
By (\ref{eq:filtration_4}),
we can take $\beta \in \rad_{-2n}(\Lambda)_-$
to be a band matrix modulo $\rad_{1-2n}(\Lambda_4)_-$,
so we have
$\phi_\beta(X) = (X-a\e)^3 X$,
for $a \in k_0^\times$.

Using Hensel's Lemma,
we can lift this to
$\Phi_\beta(X) = f_a(X) f_0(X)$
where $f_a(X)$, $f_0(X)$ are monic,
$f_a(X) \bmod{\mi_F} = (X-a\e)^3$
and $f_0(X)  \bmod{\mi_F} =  X$.
As in the proof of \cite{St3} Theorem 4.4,
when we put
$V^{a} = \ker f_a(y_\beta)$, $V^{0} = \ker f_0(y_\beta)$,
the skew stratum $[\Lambda, 2n, n, \beta]$
is split with respect to the $F$-splitting
$V = V^{a} \bot V^{0}$.

For $b \in \{ a, 0\}$,
we write 
$\Lambda^{b}(i) = \Lambda(i)\cap V^{b}$,
$i \in \Z$,
and $\beta_b = \beta|_{V^{b}}$.
Since $X$ does not divide $\phi_{\beta_a}(X)
= (X-a\e)^3$,
it follows from \cite{BK2} Proposition 3.5 that
$\beta_a \Lambda^a(i) = \Lambda^a(i-2n)$,
for $i \in \Z$.
Since $(2n, e(\Lambda)) = 2$,
we obtain
$[\Lambda^a(i):\Lambda^a(i+1)]
= [\Lambda^a(i+2): \Lambda^a(i+3)]$,
for $i \in \Z$.

Recall $\Lambda(0) = \Lambda(1)$.
So we have 
$\Lambda^a(2i) = \Lambda^a(2i+1)$, for $i \in \Z$.
Since $\Lambda(2) \neq \Lambda(3)$,
we get 
$\Lambda^0(2) \neq \Lambda^0(3)$.
This implies $\Lambda^0(3) = \p \Lambda^0(2)$ and 
$\Lambda^0(-3) = \ldots  = \Lambda^0(1) 
= \Lambda^0(2)$
because
$\dim_F V^{0} = 1$ and $e(\Lambda^0) = 6$.

\begin{lem}\label{lem:over3}
Let $[\Lambda_4, 2n, n, \beta]$ be a fundamental 
skew stratum.
Then the space $(V^0, h|_{V^0})$ represents 1.
\end{lem}
\begin{proof}
The dual lattice of $\Lambda^{0}(0)$ with respect to 
$(V^{0}, h|_{V^{0}})$ is $\Lambda^{0}(-1) = \Lambda^0(0)$.
Since $\dim_F V^{0} = 1$ and $F$ is unramified over $F_0$,
this implies that
the form $f|_{V^{0}}$ represents 1.
This completes the proof.
\end{proof}

We can apply the construction of a splitting 
to $\Lambda_5$.
The proof of the next lemma is similar to that of
Lemma~\ref{lem:over3}.
\begin{lem}\label{lem:over35}
Let $[\Lambda_5, 2n, n, \beta]$ be a fundamental 
skew stratum.
Then the space $(V^0, h|_{V^0})$ does not represent 1.
\end{lem}

\begin{prop}\label{prop:disjoint_3}
Let 
$[\Lambda_4, 2n, n, \beta]$ and $[\Lambda_5, 2n, n, \gamma]$
be fundamental skew strata.
Then there are no irreducible smooth representations of 
$G$ which contain both of them.
\end{prop}
\begin{proof}
Suppose
there is an irreducible smooth representation of 
$G$ which contain both of them.
Then
there is $g \in G$ such that
$(\beta + \rad_{-n}(\Lambda_4)_-)\cap 
\Ad(g)(\gamma  + \rad_{-n}(\Lambda_5)_-)$ is non-empty.
Let $\delta$ be an element of
$(\beta + \rad_{-n}(\Lambda_4)_-)\cap 
\Ad(g)(\gamma  + \rad_{-n}(\Lambda_5)_-)$.
Then we obtain fundamental skew strata
$[\Lambda_4, 2n, n, \delta]$ and 
$[\Lambda_5, 2n, n, \Ad(g^{-1})\delta]$.

Let $V = V^a \bot V^0$ denote the $F$-splitting 
obtained by
applying the above construction to $[\Lambda_4, 2n, n, \delta]$,
and $V = W^a \bot W^0$ the same object 
for $[\Lambda_5, 2n, n, \Ad(g^{-1})\delta]$.
Then we have $W^0 = g^{-1} V^0$.
This contradicts Lemmas~\ref{lem:over3}
and \ref{lem:over35}.
This completes the proof.
\end{proof}

Let $[\Lambda, 2n, n, \beta]$ be a fundamental skew stratum
such that $\Lambda = \Lambda_4$ or $\Lambda_5$.
We put $\Lambda'(i) = \Lambda^a(2i)$, for $i \in \Z$.
Then $\Lambda'$ is a strict  $\ri_F$-lattice sequence
in $V^a$ of period 3.
Proposition~\ref{prop:pure} says that
the stratum $[\Lambda', n, n-1, \beta_a]$ 
is simple and $E_a = F[\beta_a]$ is a
totally ramified extension of degree 3 over $F$.
It is easy to observe that
the stratum $[\Lambda_a, 2n, 2n-1, \beta_a]$
is also simple.

The equation $\phi_{\beta_0}(X) = X$
implies $\beta_0 \in \rad_{1-2n}(\Lambda^0)_-$.
We can replace $\beta_0$ with 0.
So $[\Lambda, 2n, n, \beta]$ is a skew semisimple stratum
in $A$
such that the $G$-centralizer $G_E$ of $\beta$
is a maximal compact torus.

\begin{prop}\label{prop:conj_3}
Let $\Lambda$ be $\Lambda_4$ or $\Lambda_5$.
Let 
$[\Lambda, 2n, n, \beta]$ and
$[\Lambda, 2n, n, \gamma]$ be fundamental skew strata
occurring in some irreducible smooth representation of $G$.
Then
$(P_{n+1}(\Lambda), \psi_\gamma)$ is 
a $P_{0}(\Lambda)$-conjugate of 
$(P_{n+1}(\Lambda), \psi_\beta)$.
\end{prop}
\begin{proof}
Note that Lemma~\ref{lem:recover} holds
for any fundamental skew stratum $[\Lambda, 2n, n, \beta]$
of this section because $[\Lambda^a, 2n, n, \beta_a]$ is 
maximal simple and 
$\dim_F V^0 = 1$.
This is exactly as in the proof of Proposition~\ref{prop:conj_4}.
\end{proof}

By Proposition~\ref{prop:disjoint_3} and \ref{prop:conj_3},
the set of equivalence classes of irreducible smooth representations
of $G$ of level $n/3$ is the disjoint union 
$\bigcup_\beta \Irr(G)^{(P_{n+1}(\Lambda_4), \psi_\beta)}
\cup \bigcup_\gamma \Irr(G)^{(P_{n+1}(\Lambda_5), \psi_\gamma)}$,
where $\beta$ (respectively $\gamma$) runs over 
the $P_{0}(\Lambda_4)$ (respectively $P_{0}(\Lambda_5)$)-conjugacy classes of elements in $\rad_{-2n}(\Lambda_4)_-
/\rad_{-n}(\Lambda_4)_-$
(respectively $\rad_{-2n}(\Lambda_5)_-/
\rad_{-n}(\Lambda_5)_-$)
which generate a fundamental skew stratum.
Each set in this union is classified by the results in \S \ref{sec:max}.
\begin{rem}
The restriction $[\Lambda, 2n, 2n-1, \beta]$ is not always
semisimple.
But (\ref{eq:intertwining}) holds since $\dim_F V^0 = 1$,
so we can apply the arguments in \S \ref{sec:max}.
\end{rem}

\section{Representations of half-integral level}\label{sec:h_int}
\setcounter{equation}{0}\label{Hecke}
\subsection{Semisimplification of skew strata}\label{subsec:norm}
Let $m$ be a positive integer.
By Theorem~\ref{thm:strict},
an irreducible representation $\pi$ of $G$ of level $m -1/2$
contains a fundamental skew stratum 
$[\Lambda, n, n-1, \beta]$
which satisfies one of the following conditions:
\begin{enumerate}
\item[(\ref{sec:h_int}-i)] $\Lambda = \Lambda_3$ and $n = 4m-2$;

\item[(\ref{sec:h_int}-ii)] $\Lambda = \Lambda_1$ and $n = 2m-1$;

\item[(\ref{sec:h_int}-iii)] $\Lambda = \Lambda_2$ and $n = 2m-1$.
\end{enumerate}
\noindent
We also consider skew strata $[\Lambda, n, n-1, \beta]$
with one 
of the following conditions:
\begin{enumerate}
\item[(\ref{sec:h_int}-ii')] $\Lambda = \Lambda_4$ and $n = 6m-3$;

\item[(\ref{sec:h_int}-iii')] $\Lambda = \Lambda_5$ and $n = 6m-3$.
\end{enumerate}

Let
$[\Lambda_3, 4m-2, 4m-3, \beta]$ be a fundamental skew stratum.
Up to equivalence of skew strata,
we can choose $\beta \in \rad_{2-4m}(\Lambda_3)$
as follows:
\begin{eqnarray}\label{eq:beta_h}
\beta = \p^{-m}
\left(
\begin{array}{c|cc|c}
0 & 0 & 0 & a\e\\ \hline
0 & 0 & Y & 0\\ 
0 & -\p \overline{Y} & 0 & 0\\ \hline
\p d\e & 0 & 0 & 0
\end{array}
\right), Y \in \ri_F,\ a, d \in \ri_0.
\end{eqnarray}
Then
we have 
$\phi_\beta(X) = (X -ad\varepsilon)^2 (X +Y\overline{Y})^2
\pmod{\mi_F}$.
Since we are assuming 
$[\Lambda_3, 4m-2, 4m-3, \beta]$ is fundamental,
we have
$Y\overline{Y} \in \ri_0^\times$ or $ad \in \ri_0^\times$.
We decompose this case into the following 
three cases:
\begin{enumerate}
\item[(\ref{sec:h_int}-ia)]: $ad\varepsilon \equiv -Y\overline{Y}
\pmod{\mi_F}$;

\item[(\ref{sec:h_int}-ib)] $ad\varepsilon \not\equiv -Y\overline{Y}
\pmod{\mi_F}$ and $ad \in \ri_0^\times$;

\item[(\ref{sec:h_int}-ic)] $ad \in \mi_0$.
\end{enumerate}
\begin{prop}\label{prop:reduction}
With the notation as above,
suppose that an irreducible smooth representation $\pi$ of $G$
contains a fundamental skew stratum $[\Lambda_3, 4m-2, 4m-3, \beta]$
of type (\ref{sec:h_int}-ic).
Then $\pi$ contains 
a fundamental skew strata 
of type (\ref{sec:h_int}-ii') or (\ref{sec:h_int}-iii').
\end{prop}
\begin{proof}
If $a \in \mi_0$,
then 
we have
$\beta +\rad_{3-4m}(\Lambda_3) \subset
\rad_{3-6m}(\Lambda_4)$.
Similarly,
if $d \in \mi_0$,
then $\beta +\rad_{3-4m}(\Lambda_3) \subset
\rad_{3-6m}(\Lambda_5)$.
Since the level of $\pi$ is $m-1/2$,
the lemma follows immediately.
\end{proof}

Let $[\Lambda_1, 2m-1, 2m-2, \beta]$ be a fundamental skew stratum.
We may assume that $\beta \in \rad_{1-2m}(\Lambda_1)$ 
has the following form:
\begin{eqnarray}
\beta = \p^{-m}
\left(
\begin{array}{cc|cc}
0 & \p C & 0 & 0\\
M & 0 & Y & -\overline{C} \\ \hline
0 & -\p \overline{Y} & 0 & 0 \\
0 & -\p\overline{M} & 0 & 0
\end{array}
\right),\ 
C, M, Y \in \ri_F.
\end{eqnarray}
Then we have
$\phi_\beta(X) = X^2 (X +Y\overline{Y} -CM -\overline{CM})^2$.
\begin{prop}\label{prop:reduction2}
Suppose that an irreducible smooth representation $\pi$ of $G$
contains a skew stratum $[\Lambda, n, n-1, \beta]$
of type (\ref{sec:h_int}-ii) (respectively (\ref{sec:h_int}-iii)).
Then $\pi$ contains 
a fundamental skew strata of type (\ref{sec:h_int}-ii') 
(respectively (\ref{sec:h_int}-iii')).
\end{prop}
\begin{proof}
We may replace $[\Lambda_1, 2m-1, 2m-2, \beta]$
with $[\Lambda_1, 2m-1, 2m-2, \Ad(g)\beta]$,
for $g \in P_0(\Lambda_1)$.
It is easy to observe that 
we may assume $M \equiv 0 \pmod{\mi_F}$
after $P_0(\Lambda)$-conjugation.
Then we get
$\beta +\rad_{2-2m}(\Lambda_1) \subset \rad_{3-6m}(\Lambda_4)$.
This implies that
$\pi$ contains some fundamental skew stratum
$[\Lambda_4, 6m-3, 6m-4, \gamma]$.

We can treat the case (\ref{sec:h_int}-iii)
in a similar fashion.
\end{proof}

Due to Propositions~\ref{prop:reduction} and \ref{prop:reduction2},
an irreducible smooth representation of $G$ of level $m-1/2$
contains a fundamental skew stratum of type
(\ref{sec:h_int}-ia),
(\ref{sec:h_int}-ib),
(\ref{sec:h_int}-ii'),
or (\ref{sec:h_int}-iii').

\subsection{Case (\ref{sec:h_int}-ia)}\label{u11_E}
Let $[\Lambda, n, n-1, \beta]$ be a 
fundamental skew stratum of type (\ref{sec:h_int}-ia).
Replacing $\beta$ with an element in 
$\beta + \rad_{1-n}(\Lambda)_-$,
we can assume that 
$ad \varepsilon = -Y \overline{Y}$.

Since $\beta^2 = \p^{1-2m}ad \varepsilon$,
the algebra $E = F[\beta]$ is a totally ramified extension of
degree 2 over $F$, and 
$[\Lambda, n, n-1, \beta]$ is a skew simple 
stratum in $A$.
Note that $E$ is tamely ramified over $F$
since we are assuming that $F$ is 
of odd residual characteristic.

\begin{prop}\label{prop:u11}
The group $G_E$ is isomorphic to the anisotropic unitary group
in two variables over $E_0$.
\end{prop}
\begin{proof}
The group $G_E$ is the unramified unitary in two variables 
defined by a hermitian form $(V, h_E)$ 
induced by the involutive algebra $(B, \sigma)$.
By \cite{BS} Section 5,
we may choose $h_E$ so that 
for any $\ri_E$-lattice $L$ in $V$,
the dual lattice of $L$ with respect to $h_E$
equals to that with respect to $h$.
Then $\Lambda$ is a strict self-dual $\ri_E$-lattice in $(V, h_E)$
of period 2 with $d(\Lambda) = -1$.
The assertion follows from \cite{U22} Lemma 1.11.
\end{proof}

Proposition~\ref{prop:u11} implies that 
the group $G_E$ is compact.
So Theorem~\ref{thm:Hecke} gives a classification
of the irreducible smooth representations of $G$
which contain $[\Lambda, n, n-1, \beta]$.

\subsection{Case (\ref{sec:h_int}-ib)}
\label{subsec:half}
Let $[\Lambda, n, n-1, \beta]$ be a fundamental skew stratum 
of type (\ref{sec:h_int}-ib).
Put $V^1 = F e_1 \oplus F e_4$ and 
$V^2 = Fe_2 \oplus F e_3$.
Then the stratum $[\Lambda, n, n-1, \beta]$ is 
split with respect to 
the $F$-splitting $V = V^1 \bot V^2$.

We use the notation in \S \ref{section:Moy}.
Since $\beta_1^2 = \p^{1-2m} ad \varepsilon \cdot 1_{V^1}$,
the stratum $[\Lambda^1, n, n-1, \beta_1]$ is 
simple.
If $Y \in \ri_F^\times$,
then $[\Lambda^2, n, n-1, \beta_2]$ is also simple.
Otherwise,
we can assume that $Y = 0$.
Consequently,
$[\Lambda, n, n-1, \beta]$ is skew semisimple.

\begin{prop}\label{prop:sch}
The group $G_E$ is 
compact.
\end{prop}
\begin{proof}
If $Y \in \ri_F^\times$,
then $G_E$ is isomorphic to 
$U(1)(E_1/E_{1,0}) \times U(1)(E_2/E_{2,0})$.
If not,
then $G_E$ is isomorphic to 
$U(1)(E_1/E_{1,0}) \times U(2)(F/F_{0})$,
where $U(2)$ denote the anisotropic unitary group in 
two variables.
\end{proof}

Recall that we are assuming $F$ is of odd residual characteristic.
Then $E_i$ is tamely ramified over $F$, for $i = 1, 2$.
Therefore the irreducible smooth representations of $G$
containing $[\Lambda, n, n-1, \beta]$ are classified 
by Theorem~\ref{thm:Hecke}.

\subsection{Case (\ref{sec:h_int}-ii')}\label{5_1}
Let $[\Lambda, n, n-1, \beta]$ be a fundamental skew stratum
of type (\ref{sec:h_int}-ii').
Up to equivalence class of skew strata,
we may assume that $\beta \in \rad_{-n}(\Lambda)_-$ 
has the following form:
\begin{eqnarray}\label{eq:beta1}
\beta = \p^{-m}
\left(
\begin{array}{cccc}
0 & 0 & 0 & 0\\
0 & 0 & Y & 0\\ 
0 & -\p \overline{Y} & 0 & 0\\
0 & 0 & 0 & 0
\end{array}
\right),\ Y \in \ri_F.
\end{eqnarray}

Because the stratum is fundamental, we have $Y \in \ri_F^\times$.
As in the previous section, when
we define an orthogonal $F$-splitting $V = V^1 \bot V^2$ by
$V^1 = Fe_1 \oplus Fe_4$ and 
$V^2 = Fe_2 \oplus Fe_3$,
the skew stratum $[\Lambda, n, n-1, \beta]$
is split with respect to 
$V = V^1 \bot V^2$.

We use the notation in Section \ref{section:Moy}.
Then $\beta_1 = 0$ and
$E_2 = F[\beta_2]$ is a totally ramified extension 
of degree 2 over $F$,
and hence $E_i$ is tamely ramified over $F$, for $i = 1, 2$.

We shall apply the construction in Section \ref{section:Moy}.
Since $n$ is odd,
we have $J = J_+$ and $\rho = \Psi_\beta$.
We abbreviate $\Psi = \Psi_\beta$.
\begin{thm}\label{thm:H-isom1}
With the notation as above,
there exists a support-preserving, $*$-isomorphism 
$\eta : \He(G_E//J', \Psi) 
\simeq \He(G//J, \Psi)$.
\end{thm}
\begin{rem}
The isomorphism in Theorem~\ref{thm:H-isom1}
induces a bijection from
$\Irr(G_E)^{(J', \Psi)}$ to $\Irr(G)^{(J, \Psi)}$.
The character $\Psi$ extends to a character of $G_E$,
so we get a bijection
from $\Irr(G_E)^{(J', 1)}$ to $\Irr(G)^{(J, \Psi)}$.
Since the centers of $G$ and $G_E$ are both compact,
this bijection maps a supercuspidal representation of $G_E$ to 
that of $G$.
\end{rem}

We put $B' = P_{0}(\Lambda) \cap G_E$.
Then $B'$ is the Iwahori subgroup of $G_E$ and
normalizes $J' = P_{n}(\Lambda) \cap G_E$.
We define elements $s_1$ and $s_2$ in $G_E$ by
\[
s_1 = 
\left(
\begin{array}{cccc}
0 & 0 & 0 & 1\\
0 & 1 & 0 & 0\\
0 & 0 & 1 & 0\\
1 & 0 & 0 & 0
\end{array}
\right),\
s_2 = 
\left(
\begin{array}{cccc}
0 & 0 & 0 & \p^{-1}\\ 
0 & 1 & 0 & 0\\
0 & 0 & 1 & 0\\ 
\p & 0 & 0 & 0
\end{array}
\right).
\]
Put $S = \{s_1, s_2\}$ and $W' = \langle S \rangle$.
Then we have a Bruhat decomposition $G_E = B' W' B'$.
\begin{lem}\label{lem:5c1_J'w}
Let $t$ be a non negative integer. 
Then

(i)
$[J' (s_1s_2)^t J': J'] = 
[J' (s_1s_2)^ts_1 J': J'] = 
[J'(s_2 s_1)^tJ' : J'] 
= q^{2t}$,
$[J' (s_2 s_1)^t s_2 J' : J'] = q^{2(t+1)}$.

(ii)
$[J (s_1s_2)^t J: J] = [J(s_2 s_1)^tJ : J]  =
[J (s_1s_2)^ts_1 J: J] = q^{6t}$,
$[J (s_2 s_1)^t s_2 J: J] = q^{6(t+1)}$.
\end{lem}
\begin{proof}
Let $g \in G$.
Since $[JgJ:J] = [J: J\cap gJg^{-1}]
= [\J_-: \J_-\cap g\J_- g^{-1}]$,
we can compute $[JgJ:J]$ by the description of $\{\rad_k(\Lambda)\}_{k \in \Z}$.
\end{proof}

For $\mu \in F$ and $\nu \in F^\times$,
 we set
\[
u(\mu) = 
\left(
\begin{array}{cccc}
1 & 0 & 0 & \mu\\ 
0 & 1 & 0 & 0\\
0 & 0 & 1 & 0\\ 
0 & 0 & 0 & 1
\end{array}
\right),\
\underline{u}(\mu) = 
\left(
\begin{array}{cccc}
1 & 0 & 0 & 0\\ 
0 & 1 & 0 & 0\\
0 & 0 & 1 & 0\\ 
\mu & 0 & 0 & 1
\end{array}
\right),\
h(\nu) = 
\left(
\begin{array}{cccc}
\nu & 0 & 0 & 0\\ 
0 & 1 & 0 & 0\\
0 & 0 & 1 & 0\\ 
0 & 0 & 0 & \overline{\nu}^{-1} 
\end{array}
\right).
\]
For $g \in G_E$,
let $e_g$ denote the element in 
$\He(G_E//J', \Psi)$
such that
 $e_g(g) = 1$ and $\mathrm{supp}(e_g) = J'gJ'$.
\begin{thm}\label{thm:H_G'}
Suppose $m\geq 2$.
Then the algebra $\He(G'//J', \Psi)$ is generated by the elements $e_g$, $g \in B' \cup S$.
These elements are subject to the following relations:

(i) $e_k = \Psi(k) e_1$, $k \in J'$,

(ii) $e_k * e_{k'} = e_{kk'}$, $k, k' \in B'$,

(iii) $e_k * e_{s} = e_s * e_{sks}$, $s \in S$, $k \in B' \cap sB's$,

(iv) $e_{s_1}*e_{s_1} = e_1$,

$e_{s_2} * e_{s_2} = [J's_2 J' :J']\sum_{x \in \ri_0/\mi_0^2} 
e_{u(\p^{m-2}x\e)}$,

(v) For $\mu \in \ri_0^\times \e$,

$e_{s_1}* e_{u(\mu)}* e_{s_1}
= e_{u(\mu^{-1})} * e_{s_1} * e_{h(\mu)} * e_{u(\mu^{-1})}$,

$e_{s_2}* e_{\underline{u}(\p \mu)}* e_{s_2}
= [J's_2J':J']
e_{\underline{u}(\p \mu^{-1})} * e_{s_2} * e_{h(-\mu^{-1})} 
* e_{\underline{u}(\p \mu^{-1})}$.
\end{thm}
\begin{proof}
The proof is very similar to that of
\cite{Harish} Chapter 3 Theorem 2.1.
\end{proof}

For $g \in G_E$,
we denote by
 $f_g$ the element in $\He(G//J, \Psi)$ such that
$f_g(g) = 1$ and $\mathrm{supp}(f_g) = JgJ$.
\begin{thm}\label{thm:H_G}
Suppose $m \geq 2$.
Then the algebra
$\He(G//J, \Psi)$ is generated by $f_g$, $g \in B' \cup S$
and satisfies the following relations:

(i) $f_k = \Psi(k)f_1$, $k \in J'$,

(ii) $f_k * f_{k'} = f_{kk'}$, $k, k' \in B'$,

(iii) $f_k * f_{s} = f_s * f_{sks}$, $s \in S$, $k \in B' \cap sB's$,

(iv) $f_{s_1}*f_{s_1} = f_1$,

$f_{s_2} * f_{s_2} = [Js_2 J :J]\sum_{x \in \ri_0/\mi_0^2} 
f_{u(\p^{m-2}x\e)}$,

(v) For $\mu \in \ri_0^\times \e$,

$f_{s_1}* f_{u(\mu)}* f_{s_1}
= f_{u(\mu^{-1})} * f_{s_1} * f_{h(\mu)} * f_{u(\mu^{-1})}$,

$f_{s_2}* f_{\underline{u}(\p \mu)}* f_{s_2}
= q^4 f_{\underline{u}(\p \mu^{-1})} * f_{s_2} * f_{h(-\mu^{-1})} 
* f_{\underline{u}(\p \mu^{-1})}$.
\end{thm}
\begin{proof}
Recall that if $x,  y \in G_E$ satisfy $[JxJ:J] 
 [JyJ:J] =[JxyJ:J]$,
 then we have
$JxJyJ = JxyJ$ and
$f_x* f_y = f_{xy}$.

By Proposition~\ref{prop:E_surj}, 
$\He(G//J, \Psi)$ is linearly spanned by 
$f_g$, $g \in G_E$.
For $g \in G_E$,
we write $g = b_1 w b_2$ where $b_1, b_2 \in B'$ and
$w \in W'$.
Then we have $f_g = f_{b_1}*f_w*f_{b_2}$ because
$B'$ normalizes $J$.
Let $w = s_{i_1} s_{i_2}\cdots s_{i_l}$ be a minimal expression
for $w$
with $s_{i_j} \in S$.
It follows from Lemma~\ref{lem:5c1_J'w}
that $f_w = f_{s_{i_1}}*f_{s_{i_2}}*\cdots *f_{s_{i_l}}$.
Therefore 
$\He(G//J, \Psi)$ is generated by $f_g$, $g \in B' \cup S$.

Relations (i), (ii), (iii) are trivial.
Since $[Js_1J: J] = 1$,
the relations (iv) and (v) on $s_1$ are obvious.
We shall prove relations (iv) and (v) on $s_2$ in the case
when $m = 2k$, $k \geq 1$.
The case when $m = 2k+1$, $k \geq 1$ can be treated in a similar fashion.

We will abbreviate $s = s_2$.
We can choose a common system of 
representatives for $J/J\cap sJs$
and $J\cap s Js\backslash J$ to be
\begin{eqnarray*} 
& x(a, A, B) = 
\left(
\begin{array}{cccc}
1 & 0 & 0 & 0\\
\p^k A & 1 & 0 & 0\\
\p^{k}B & 0 & 1 & 0\\
\p^m a\e -
\p^m(\p A\overline{A} + B\overline{B})/2 & -\p^{k+1}\overline{A} & -\p^k \overline{B} & 1
\end{array}
\right),\ 
\end{eqnarray*}
where $A, B \in \ri_F/\mi_F$ and  $a \in \ri_0/\mi_0^2$.
Note that $x(a,A, B)$ lies in the kernel of $\Psi$.

(iv)
For $g \in G$,
let $\delta_g$ denote the unit point mass at $g$.
As in the proof of the relation (c) in \cite{U21} Theorem 2.7,
we obtain
\begin{eqnarray*}
f_s * f_s
 & = & [JsJ:J] 
 \sum_{x \in J\cap s Js\backslash J} \Psi^{-1}(x) f_1*\delta_{sxs}*f_1.
\end{eqnarray*}

Recall that for $g \in G$,
\begin{eqnarray*}
f_1*\delta_g *f_1
= 
\left\{
\begin{array}{cl}
[JgJ:J]^{-1} f_g, & \mathrm{if}\ g \in I_G(\Psi),\\
0, & \mathrm{if}\ g \not\in I_G(\Psi).
\end{array}
\right.
\end{eqnarray*}

Observe that 
the $B$-component of $sx(a, A, B)s$ lies in 
$\rad_{0}(\Lambda)$.
By Lemma~\ref{lem:intertwine_J},
if $sx(a,A, B)s \in I_G(\Psi) = JG_E J$,
then
the $B^\bot$-component of $sx(a, A, B)s$ lies in 
$\rad_{[n/2]+1}(\Lambda)$.
This implies $A \equiv B \equiv 0$.

So we obtain
\begin{eqnarray*}
f_s * f_s
 & = & 
 [JsJ:J] 
 \sum_{a\in \ri_0/\mi_0^2} 
 f_1*\delta_{sx(a, 0, 0)s}*f_1\\
& = &
 [Js J :J]\sum_{a \in \ri_0/\mi_0^2} 
f_1*\delta_{u(\p^{m-2}a\e)}*f_1\\
& = &
[Js J :J]\sum_{a \in \ri_0/\mi_0^2} 
f_{u(\p^{m-2}a\e)}.
\end{eqnarray*}

(v)
Let $\mu \in \ri_0^\times \e$.
Put $u = \underline{u}(\p \mu) \in B'$.
As in the proof of the relation (d) in \cite{U21} Theorem 2.7,
we get
\begin{eqnarray*}
f_s * f_u*f_s
 & = & [JsJ:J] 
 \sum_{x \in J\cap s Js\backslash J} \Psi^{-1}(x) f_1*\delta_{sxus}*f_1.
\end{eqnarray*}

Put $x = x(0, A, B)$,
$\nu = \mu + \p^{m-1} a\e$,
$v = \underline{u}(\p \nu^{-1})$, and $h(-\nu^{-1})$.
Then we have
\begin{eqnarray*}
sx(a, A, B)us  
= sx \underline{u}(\p \nu) s
= sxs vshv
= [sxs, v]vshv (hv)^{-1}x(hv).
\end{eqnarray*}
Since $hv \in B'$,
the element $(hv)^{-1}x(hv)$ lies in the kernel of $\Psi$.
Since $v \in P_4(\Lambda)$ and $sxs \in P_{6k-5}(\Lambda)$,
we have $[sxs, v] \in P_{6k-1}(\Lambda) = P_{[n/2]+1}(\Lambda)$.
Observe that $[sxs, v]$ is equivalent to 
\begin{eqnarray}
1 +
\left(
\begin{array}{cccc}
-\p^{m-1} B\overline{B}\nu^{-1}/2 & 0 & 0 & 
\p^{2m-3}(B\overline{B})^2\nu^{-2}/4\\
0 & -\p^m A\overline{A} \nu^{-1} & -\p^{m-1}A\overline{B}\nu^{-1} & 0\\
0 & -\p^m \overline{A}B \nu^{-1} & -\p^{m-1} B\overline{B} \nu^{-1} & 0\\
0 & 0 & 0 & -\p^{m-1} B\overline{B}\nu^{-1}/2
\end{array}
\right)
\end{eqnarray}
modulo $\rad_{n+1}(\Lambda) + B^\bot$.
Since $P_{[n/2]+1}(\Lambda)/P_{n+1}(\Lambda)$ is abelian,
there is an element $p(B) \in P_{[n/2]+1}$ which is 
equivalent to 
\begin{eqnarray}
1 +
\left(
\begin{array}{cccc}
-\p^{m-1} B\overline{B}\nu^{-1}/2 & 0 & 0 & 
\p^{2m-3}(B\overline{B})^2\nu^{-2}/4\\
0 & 0 & 0 & 0\\
0 & 0 & -\p^{m-1} B\overline{B} \nu^{-1} & 0\\
0 & 0 & 0 & -\p^{m-1} B\overline{B}\nu^{-1}/2
\end{array}
\right)
\end{eqnarray}
modulo $\rad_{n+1}(\Lambda)$.
Thus we have $[sxs, v]p(B)^{-1} \in P_{[n/2]+1}(\Lambda)$
and $[sxs, v]p(B)^{-1} \equiv 1 \pmod{\rad_{n}(\Lambda) + B^\bot}$.
This implies $[sxs, v]p(B)^{-1}$ lies in $J$.
So we have
\begin{eqnarray*}
f_1 * \delta_{sx(a, A, B)u s}* f_1 & = & 
\Psi^{-1}([sxs, v]p(B)^{-1}) f_1*\delta_{p(B)vshv} * f_1\\
& = & 
\psi_0(2(A\overline{BY}-\overline{A}BY)\nu^{-1})
f_1*\delta_{p(B)vshv} * f_1\\
& = & 
\psi_0(\mathrm{tr}_{F/F_0}(2A\overline{BY}\nu^{-1}))
f_1*\delta_{p(B)vshv} * f_1.
\end{eqnarray*}

We therefore have
\begin{eqnarray*}
f_s * f_u*f_s & = & 
[JsJ:J] \sum_{a \in \ri_0/\mi_0^2,\ A, B\in \ri_F/\mi_F} 
\psi_0(\mathrm{tr}_{F/F_0}(2A\overline{BY}\nu^{-1}))
f_1*\delta_{p(B)vshv} * f_1
\\
& = & 
[JsJ:J] \sum_{a\in \ri_0/\mi_0^2,\ A \in \ri_F/\mi_F} 
f_1*\delta_{vshv} * f_1\\
& = &
q^2 [JsJ:J] 
 \sum_{a\in \ri_0/\mi_0^2} 
 f_1*\delta_{vshv}*f_1\\
 & = &
q^2 
 \sum_{a\in \ri_0/\mi_0^2} 
 f_{vshv}\\
 & = & 
q^4  f_{\underline{u}(\p \mu^{-1})} * f_{s_2} * f_{h(-\mu^{-1})} 
* f_{\underline{u}(\p \mu^{-1})}.
\end{eqnarray*}
\end{proof}

\begin{rem}
If $m = 1$,
the algebra
$\He(G_E//J', \Psi)$ is generated by the elements $e_g$, $g \in B' \cup S$.
In this case,
these elements are subject to the following relations:

(i) $e_k = \Psi(k)e_1$, $k \in J'$,

(ii) $e_k * e_{k'} = e_{kk'}$, $k, k' \in B'$,

(iii) $e_k * e_{s} = e_s * e_{sks}$, $s \in S$, $k \in B' \cap sB'S$,

(iv) $e_{s_1}*e_{s_1} = e_1$,

$e_{s_2} * e_{s_2} = 
(\sum_{y \in \ri_0\e/\mi_0\e,\ y \not \equiv 0}e_{s_2}
*e_{h(-y^{-1})} + q^2 e_1)(\sum_{x \in \ri_0/\mi_0} 
e_{u(x\e)})$,

(v) For $\mu \in \ri_0^\times \e$,

$e_{s_1}* e_{u(\mu)}* e_{s_1}
= e_{u(\mu^{-1})} * e_{s_1} * e_{h(\mu)} * e_{u(\mu^{-1})}$.

\noindent
We can easily see that the analogue of  Theorem~\ref{thm:H-isom1}
holds as well.
We omit the details.
\end{rem}

By Theorems~\ref{thm:H_G'} and \ref{thm:H_G},
the map
\[
\eta(e_{s_1}) =  f_{s_1},\
\eta(e_{s_2}) = q^{-2} f_{s_2},\
\eta(e_b) = f_b,\ b \in B'
\]
induces an algebra homomorphism $\eta$.
Observe that
$\eta(e_g) = (\mathrm{vol}(J'gJ')/\mathrm{vol}(JgJ))^{1/2}f_g$, for $g \in G_E$.
Thus Propositions~\ref{prop:E_surj} and \ref{prop:H_inj}
imply that $\eta$ is a bijection.
The $*$-preservation of $\eta$ follows from
the relations
$e_g^* = e_{g^{-1}}$ and $f_g^* = f_{g^{-1}}$,
for $g \in G_E$.

\subsection{Case (\ref{sec:h_int}-iii')}\label{5_2}
Let $[\Lambda, n, n-1, \beta]$ be a fundamental skew stratum 
of type (\ref{sec:h_int}-iii').
Put
\begin{eqnarray}
t = 
\left(
\begin{array}{cccc}
0 & 0 & 0 & 1\\
0 & 0 & 1 & 0\\
0 & \p & 0 & 0\\
\p & 0 & 0 & 0
\end{array}
\right).
\end{eqnarray}
Then $t$ is a similitude on $(V, f)$,
and hence we can consider the action of $t$ on the set of skew 
strata in $A$.
Observe that $t\Lambda$ is a translate of $\Lambda_4$.
So the stratum $[\Lambda, n, n-1, \beta]$ is a
$t$-conjugate of a stratum of type 
(\ref{sec:h_int}-ii').
Replacing the objects in case (\ref{sec:h_int}-ii')
with the $t$-conjugate of them,
we get a classification of 
the irreducible smooth representations of $G$
which contain $[\Lambda, n, n-1, \beta]$.

\subsection{Intertwining problems}
In this section,
we consider the condition when 
an irreducible smooth representation $\pi$ of $G$ contains
two skew 
skew strata of type
(\ref{sec:h_int}-ia),
(\ref{sec:h_int}-ib),
(\ref{sec:h_int}-ii'),
or (\ref{sec:h_int}-iii').

\begin{prop}
Let $[\Lambda_3, 4m-2, 4m-3, \beta]$ be a 
skew stratum of type (\ref{sec:h_int}-ia)
and let $[\Lambda, n, n-1, \gamma]$ be a skew stratum 
of type (\ref{sec:h_int}-ia),
(\ref{sec:h_int}-ib),
(\ref{sec:h_int}-ii'),
or (\ref{sec:h_int}-iii').
Suppose that there is an irreducible smooth representation of $G$
which contains both of them.
Then $[\Lambda, n, n-1, \gamma]$ is also of type (\ref{sec:h_int}-ia)
and $(P_{n}(\Lambda), \psi_\gamma)$ is 
a $P_0(\Lambda_3)$-conjugate of $(P_{4m-2}(\Lambda_3), \psi_\beta)$.
\end{prop}
\begin{proof}
The assumption implies that $\phi_\beta(X) = \phi_\gamma(X)$,
and hence $[\Lambda, n, n-1, \beta]$ is also of type 
(\ref{sec:h_int}-ia).
By the uniqueness of level,
we have $n = 4m-2$.
Take $\beta$ and $\gamma$
as in (\ref{eq:beta_h}).
Then it is easy to see that
$\gamma$ is a $P_{0}(\Lambda_3)$-conjugate of $\beta$.
\end{proof}

\begin{prop}
Let $[\Lambda_3, 4m-2, 4m-3, \beta]$ be a 
skew stratum of type (\ref{sec:h_int}-ib)
and let $[\Lambda, n, n-1, \gamma]$ be a skew stratum 
of type (\ref{sec:h_int}-ia),
(\ref{sec:h_int}-ib),
(\ref{sec:h_int}-ii'),
or (\ref{sec:h_int}-iii').
Suppose that there is an irreducible smooth representation of $G$
which contains both of them.
Then $[\Lambda, n, n-1, \gamma]$ is also of type (\ref{sec:h_int}-ib)
and $(P_{n}(\Lambda), \psi_\gamma)$ is 
a $P_0(\Lambda_3)$-conjugate of $(P_{4m-2}(\Lambda_3), \psi_\beta)$.
\end{prop}
\begin{proof}
Since $\phi_\beta(X) = (X-ad\varepsilon)^2(X+Y\overline{Y})^2
\pmod{\mi_F}$ and $ad \varepsilon \not \equiv -Y\overline{Y}
\pmod{\mi_F}$,
the stratum
$[\Lambda, n, n-1, \gamma]$ is of type (\ref{sec:h_int}-ib),
(\ref{sec:h_int}-ii'),
or (\ref{sec:h_int}-iii').
By the assumption, we have $m-1/2 = n/e(\Lambda)$,
so 
we can take $\gamma$ to be 
\begin{eqnarray}
\gamma 
= 
\p^{-m}
\left(
\begin{array}{cccc}
0 & 0 & 0 & b\e\\
0 & 0 & Z & 0\\
0 & -\p \overline{Z} & 0 & 0\\
\p c\e & 0 & 0 & 0
\end{array}
\right),\ Z \in \ri_F,\ b, c \in \ri_0.
\end{eqnarray}
By the uniqueness of characteristic polynomial,
we obtain 
$(X-ad\varepsilon)^2(X+Y\overline{Y})^2
 = (X-bc\varepsilon)^2(X+Z\overline{Z})^2
\pmod{\mi_F}$.

Put $V^1 = F e_1 \oplus F e_4$ and 
$V^2 = Fe_2 \oplus F e_3$.
Then the strata $[\Lambda_3, 4m-2, 4m-3, \beta]$
and  $[\Lambda, n, n-1, \gamma]$ 
are 
split with respect to 
$V = V^1 \bot V^2$.

Put $\mathcal{M} = A^{11} \oplus A^{22}$.
By Proposition~\ref{prop:5c1_surj},
there is an element $g \in G$
such that 
$(\beta + \rad_{3-4m}(\Lambda_3)_-
\cap \mathcal{M})\cap \Ad(g)(\gamma
+ \rad_{1-n}(\Lambda)_-\cap \mathcal{M}) \neq \emptyset$.
Take $\delta$ in this intersection.
Then we get two skew strata 
$[\Lambda_3, n, n-1, \delta]$ and 
$[\Lambda, n, n-1, \Ad(g^{-1})(\delta)]$.

By Hensel's lemma,
there are monic polynomials $f_1(X), f_2(X)$ 
in $\ri_F[X]$ such that 
$f_1(X) \equiv (X-ad\varepsilon)^2 \pmod{\mi_F}$,
$f_2(X) \equiv (X+Y\overline{Y})^2 \pmod{\mi_F}$,
and 
$\Phi_\delta(X) = f_1(X) \cdot f_2(X)$.
Because we take $\delta$ in $\mathcal{M}$,
we get $\ker f_1(y_\delta) = V^1$ and 
$\ker f_2(y_\delta) = V^2$.

Similarly,
there are monic polynomials $g_1(X), g_2(X)$ 
in $\ri_F[X]$ such that 
$g_1(X) \equiv (X-bc\varepsilon)^2 \pmod{\mi_F}$,
$g_2(X) \equiv (X+Z\overline{Z})^2 \pmod{\mi_F}$,
and 
$\Phi_{\Ad(g^{-1})(\delta)}= 
g_1(X) \cdot g_2(X)$.
We get $\ker g_1(y_\delta) = g V^1$ and 
$\ker g_2(y_\delta) = g V^2$.

Since $\Phi_\delta(X) = \Phi_{\Ad(g^{-1})(\delta)}(X)$,
we see that $f_1(X) = g_1(X)$ or $g_2(X)$,
and hence $V^1 = \ker g_1(y_\delta)$ or 
$\ker g_2(y_\delta)$.
Since $V^1$ is isotropic and 
$V^2$ is anisotropic,
we have $V^1 =  g V^1$  
and $V^2 = g V^2$.
So we conclude $g \in G\cap \mathcal{M}$.

As in the proof of the uniqueness of the characteristic polynomial,
we have $ad \varepsilon \equiv bc \varepsilon$ and 
$Y \overline{Y} \equiv Z\overline{Z}$ modulo $\mi_F$.
This implies that $[\Lambda, n, n-1, \gamma]$ is also 
of type (\ref{sec:h_int}-ib)
and a $P_{0}(\Lambda_3)$-conjugate of $[\Lambda_3, 4m-2, 4m-3, \beta]$.
\end{proof}

\begin{rem}
Let $n = 6m-3$
and let $\beta$ be as in (\ref{eq:beta1}).

(i)
We claim that a skew stratum
$[\Lambda_4, n, n-1, \beta]$ 
is not $G$-conjugate of 
$[\Lambda_5, n, n-1, \beta]$.

Let $\Lambda$ be a self-dual $\ri_F$-lattice sequence.
From the group of $P_{n}(\Lambda)$,
we get $\rad_{n}(\Lambda)_-$ as the image of Cayley map
$x \mapsto (1+x)(1-x)^{-1}$.
Since $F/F_0$ is unramified,
we get $\rad_{n}(\Lambda) = \rad_{n}(\Lambda)_-
\oplus \e \rad_{n}(\Lambda)_-$.
So we can recover $\rad_n(\Lambda)$ and $\rad_{1-n}(\Lambda) = \rad_n(\Lambda)^*$ from $P_{n}(\Lambda)$.

Return to the skew stratum $[\Lambda_4, n, n-1, \beta]$.
By the periodicity of $\Lambda$,
we can recover $\rad_3(\Lambda_4)$ and $\rad_4(\Lambda_4)$.
Observe that 
$\Lambda_4(\Z)$ consists of all $\ri_F$-lattices 
of the form
$\rad_{k}(\Lambda_4) \cdot L$,
where $L$ is an $\ri_F$-lattice in $V$ and $k = 3, 4$.
The sequence $\Lambda_5$ has the same property.

Suppose that there is $g \in G$ so that $P_n(\Lambda_4)
= Ad(g)(P_n(\Lambda_5))$.
Then we have $\Lambda_5(\Z) =  \Lambda_4(\Z)
= g \Lambda_5(\Z)$.
This implies $g \in P_0(\Lambda_5)$
and hence $P_n(\Lambda_4) = P_n(\Lambda_5)$.
This contradict the fact $\rad_n(\Lambda_4) \neq \rad_n(\Lambda_5)$.

(ii)
By Proposition~\ref{prop:reduction},
if an irreducible smooth representation $\pi$ of $G$
contains $[\Lambda_3, 4m-2, 4m-3, \beta]$,
then $\pi$ contains both of 
$[\Lambda_4, n, n-1, \beta]$ 
and
$[\Lambda_5, n, n-1, \beta]$.
\end{rem}

\section{Comparison}
\setcounter{equation}{0}\label{comparison}
There is another unramified unitary group in four variables 
defined over $F_0$ denoted by $\rU(2,2)$,
which is quasi-split and an inner form of the non quasi-split $\rU(4)$.
In this section,
we compare the irreducible smooth representations of non quasi-split $\rU(4)$ with $\rU(2,2)$ of non-integral level.

Let $G$ denote the non quasi-split $\rU(4)$ or
unramified $\rU(2,2)$ defined over $F_0$.
The results on $\rU(2,2)$ analogous to this paper
can be found in \cite{U22}.

We list up the characteristic polynomials $\phi_\beta(X)$
and the form of the groups $G_E$ for 
semisimple skew strata  $[\Lambda, m, m-1, \beta]$ of $G$
of non-integral level in this paper and \cite{U22}.
The level of a fundamental skew stratum of $G$
should be $n$, $n/2$, $n/3$ or $n/4$, for some positive 
integer $n$.

(i) level $n/4$:
The characteristic polynomial has the form 
$\phi_\beta(X) = (X-a)^4$, for $a \in k_0^\times$.
The algebra $E = F[\beta]$ is a totally ramified extension 
of degree 4 over $F$
and $G_E$ is isomorphic to the unramified unitary group
$\rU(1)(E/E_0)$.

(ii) level $n/3$:
The characteristic polynomial is of the form 
$\phi_\beta(X) = (X-a\e)^3$X, for $a \in k_0^\times$.
The algebra $E = F[\beta]$ is isomorphic to $E_1 \oplus F$,
where $E_1$ is a totally ramified extension 
of degree 3 over $F$.
The group $G_E$ is isomorphic to a product of 
unramified unitary groups
$\rU(1)(E_1/E_{1,0}) \times \rU(1)(F/F_0)$.

(iii) level $n/2$:
The characteristic polynomial $\phi_\beta(X)$ has one 
of the following form:

\begin{enumerate}
\item[(iii-a)] $(X-a)^4$, for $a \in k_0^\times$;

\item[(iii-b)] $(X-a)^2(X-b)^2$, for $a, b \in k_0^\times$ such that
$a \neq b$.

\item[(iii-c)]
$(X-a)^2 X^2$, for $a \in k_0^\times$.
\end{enumerate}

Case (iii-a):
The stratum is simple and $E$ is a quadratic ramified 
extension over $F$.
If $G = \rU(2,2)$,
then $G_E$ is isomorphic to $\rU(1,1)(E/E_0)$.
If $G$ is not quasi-split,
then $G_E$ is isomorphic to $\rU(2)(E/E_0)$.

Case (iii-b):
The algebra $E$ is isomorphic to $E_1 \oplus E_2$,
where $E_i$ is a quadratic ramified 
extension over $F$, for $i = 1, 2$.
The group $G_E$ is isomorphic to $\rU(1)(E_1/E_{1,0})
\times \rU(1)(E_2/E_{2,0})$.

Case (iii-c):
The algebra $E$ is isomorphic to $E_1 \oplus F$,
where $E_1$ is a quadratic ramified 
extension over $F$.
The group $G_E$ is isomorphic to $\rU(1)(E_1/E_{1,0})
\times \rU(1,1)(F/F_0)$
or $\rU(1)(E_1/E_{1,0})
\times \rU(2)(F/F_0)$.

\begin{rem}
The difference is only case (iii-a).
We can see that the set of the irreducible supercuspidal representations of
$\rU(1,1)(E/E_0)$ is very close to that of 
$\rU(2)(E/E_0)$
by establishing Hecke algebra isomorphisms for those groups.
\end{rem}

When $G$ is not quasi-split,
an irreducible smooth representation of $G$ of non-integral level
contains one of skew semisimple strata listed above.
For $\rU(2,2)$,
we need to consider semisimple (but 
not skew semisimple), skew strata of the following type:

(iii-d) The level is half-integral and 
the characteristic polynomial $\phi_\beta(X)$ has 
the form
$(X-\lambda)^2 (X-\overline{\lambda})^2$, for 
$\lambda \in k_F^\times$, $\lambda \neq \overline{\lambda}$.
In this case,
$E = E_1 \oplus E_2$,
where $E_i$ is quadratic ramified for $i = 1, 2$,
and $G_E$ is isomorphic to $\mathrm{GL}_1(E_1)$.

\begin{rem}
A stratum of case (iii-d) is called $G$-split
in \cite{St3}.
It follows from the proof of \cite{St3} Theorem 3.6
that 
if an irreducible smooth representations $\pi$ of $G = \rU(2,2)(F/F_0)$
contains a stratum of type (iii-d),
then there is a parabolic subgroup $P$
of $G$ whose Levi component 
is isomorphic to $\mathrm{GL}_2(F)$,
such that the Jacquet module of $\pi$ relative to $P$
is not zero.
This difference yields from the fact that 
the non quasi-split $\rU(4)$ has no parabolic subgroups of 
such type.
\end{rem}

We close this paper with a table of skew semisimple stratum
$[\Lambda,n, r, \beta]$
for the non quasi-split $\rU(4)$
we have considered.
\begin{center}
\begin{tabular}{|c|c|c|c|c|}
\hline
$\Lambda$ & $n$ & $r$ & $\beta$ & section \\
\hline
$\Lambda_3$ & $(n, 4) = 1$ & $[n/2]$ & fundamental & \S~\ref{sec:over4} \\
\hline
$\Lambda_4$ & $(n, 6) = 2$ & $[n/2]$ & fundamental & 
\S~\ref{sec:r_3}\\
\hline
$\Lambda_5$ & $(n, 6) = 2$ & $[n/2]$& fundamental & 
\S~\ref{sec:r_3}\\
\hline
$\Lambda_3$ & $(n, 4) = 2$ & $n-1$ & (\ref{eq:beta_h}),
$ad\varepsilon = -Y\overline{Y}$ & 
\S~\ref{u11_E}\\
\hline
$\Lambda_3$ & $(n, 4) = 2$ & $n-1$ & (\ref{eq:beta_h}),
$ad\varepsilon \not\equiv -Y\overline{Y}$,
$ad \not\equiv 0$&  \S~\ref{subsec:half}\\
\hline
$\Lambda_4$ & $(n, 6) = 3$ & $n-1$ & (\ref{eq:beta1})& 
\S~\ref{5_1}\\
\hline
$\Lambda_5$ & $(n, 6) = 3$ & $n-1$ & (\ref{eq:beta1}) & 
\S~\ref{5_2}\\
\hline
\end{tabular}
\end{center}

\end{document}